\documentclass[11pt,english,a4paper]{amsart}
\usepackage[utf8]{inputenc}
\usepackage[T1]{fontenc}
\usepackage[english]{babel}
\usepackage{verbatim}

\usepackage{xcolor}
\usepackage{amsmath}
\usepackage{amssymb}
\usepackage{amsthm}
\usepackage{lmodern}
\usepackage{latexsym}
\usepackage{amsfonts}
\usepackage{mathrsfs}
\usepackage[all]{xy}
\usepackage{bm}
\usepackage{yfonts}
\newcommand{\ind}{\mathbf{1}}
\usepackage{hyperref}
\hypersetup{pdftoolbar=true, pdftitle={A sharp almost sure upper bound for partial sums of random multiplicative functions}, pdffitwindow=true, colorlinks=true, citecolor=blue, filecolor=black, linkcolor=red, urlcolor=red, hypertexnames=false}
\numberwithin{equation}{section}

\usepackage[margin=1in]{geometry}
\makeatletter\@secpenalty=-50\makeatother

\theoremstyle{plain}
\newtheorem{thm}{Theorem}
\newtheorem{lem}{Lemma}[section]
\newtheorem{prop}[lem]{Proposition}
\newtheorem{cor}[lem]{Corollary}

\theoremstyle{definition}

\theoremstyle{remark}

\newcommand{\Z}{\mathbb{Z}}

\newcommand{\R}{\mathbb{R}}

\newcommand{\T}{\mathbb{T}}
\renewcommand{\P}{\mathbb{P}}
\newcommand{\E}{\mathbb{E}}

\newcommand{\CB}{\mathcal{B}}
\newcommand{\CF}{\mathcal{F}}
\newcommand{\CG}{\mathcal{G}}
\newcommand{\CH}{\mathcal{H}}
\newcommand{\CP}{\mathcal{P}}
\newcommand{\CQ}{\mathcal{Q}}
\newcommand{\CS}{\mathcal{S}}
\newcommand{\CX}{\mathcal{X}}

\usepackage{url}

\newcommand{\dee}{\mathrm{d}}

\newcommand{\eps}{\varepsilon}
\renewcommand{\epsilon}{\varepsilon}
\renewcommand{\phi}{\varphi}
\renewcommand{\bar}[1]{\overline{#1}}

\renewcommand{\le}{\leqslant}

\renewcommand{\ge}{\geqslant}

\begin{document}

\title[Sharp almost sure upper bound for random multiplicative functions]{A sharp almost sure upper bound for partial sums of random multiplicative functions} \author[B. Durkan]{Benjamin Durkan} \address{Department of Mathematics, University of Manchester, Manchester, M13 9PL, United Kingdom} \email{benjamin.durkan@postgrad.manchester.ac.uk}

\author[A. Pearce-Crump]{Andrew Pearce-Crump} \address{School of Mathematics, University of Bristol, Fry Building, Woodland Road, Bristol, BS8 1UG, United Kingdom} \email{andrew.pearce-crump@bristol.ac.uk}

 \date{}

\begin{abstract}
We prove that, for either a Steinhaus random multiplicative function or a Rademacher random multiplicative function $f$, and every $\eps>0$, almost surely
$$
\left|\sum_{n\le x}f(n)\right|\ll_{\eps,f}\sqrt{x}(\log\log x)^{1/4+\eps}.
$$
Together with Harper's almost sure lower bound, this determines the sharp logarithmic exponent in both models. This settles Harper's conjecture on the large fluctuations of random multiplicative functions.
\end{abstract}

\maketitle
\setlength{\abovedisplayshortskip}{\abovedisplayskip}
\setlength{\belowdisplayshortskip}{\belowdisplayskip}

\section{Introduction}\label{sec:intro}

Random multiplicative functions are probabilistic models for multiplicative arithmetic functions.  The values at the primes are independent random variables, and the values at the other positive integers are determined by multiplicativity.  In the Steinhaus model, which is a model for Dirichlet characters $\chi(n)$ and for the functions $n^{it}$, the variables $f(p)$ are uniformly distributed on the unit circle $\T$, and $f$ is extended completely multiplicatively, so that
$$
        f(n)=\prod_{p^a\Vert n}f(p)^a.
$$

In the Rademacher model, the variables $f(p)$ take the values $\pm1$ with probability $1/2$ each, and $f$ is defined by
\begin{equation}\label{eq:rademacher-definition}
        f(n)=\mu^2(n)\prod_{p\mid n}f(p),
\end{equation}
so that $f$ is supported on the squarefree integers.  This model was introduced by Wintner \cite{Wintner} as a probabilistic model for the M\"obius function $\mu(n)$.  In either model we write
$$
        M_f(x)=\sum_{n\le x}f(n),
$$
and we study the size of $M_f(x)$ as $x$ tends to infinity.

Harper \cite{HarperFluct} proved, in both models, that for every function $V(x)$ tending to infinity, almost surely there are arbitrarily large $x$ such that
\begin{equation}\label{eq:harper-lower}
        |M_f(x)|\ge
        \frac{\sqrt{x}(\log\log x)^{1/4}}{V(x)}.
\end{equation}
He conjectured that the matching upper bound holds with an additional factor $(\log\log x)^{\eps}$, for every $\eps>0$.  We prove this conjecture in both models.

\begin{thm}\label{thm:main}
Let $f$ be either a Steinhaus random multiplicative function or a Rademacher random multiplicative function.  For every $\eps>0$, almost surely
$$
        |M_f(x)|\ll_{\eps,f}
        \sqrt{x}(\log\log x)^{1/4+\eps}.
$$
\end{thm}

Together with \eqref{eq:harper-lower}, Theorem \ref{thm:main} determines the exponent of $\log\log x$ in the almost sure size of $M_f(x)$ in both models.

\begin{cor}\label{cor:almost-sure-exponent}
Let $f$ be either a Steinhaus random multiplicative function or a Rademacher random multiplicative function.  Then, almost surely,
$$
        \limsup_{x\to\infty}
        \frac{\log(1+|M_f(x)|/\sqrt{x})}
        {\log\log\log x}
        =\frac14.
$$
\end{cor}

\begin{proof}
Apply Theorem \ref{thm:main} on the probability-one intersection corresponding to $\eps=1/m$, $m\ge1$, to obtain the upper bound.  For the reverse inequality, take $V(x)=(\log\log x)^\delta$ in \eqref{eq:harper-lower}, where $0<\delta<1/4$.  Almost surely, along an unbounded sequence,
$$
        \frac{|M_f(x)|}{\sqrt{x}}
        \ge(\log\log x)^{1/4-\delta},
$$
and hence the limsup is at least $1/4-\delta$.  Letting $\delta\downarrow0$ through a countable sequence proves the result.
\end{proof}

\subsection{Typical size and limiting distribution}\label{sec:typical}

The values $f(n)$ are far from independent, since they are all determined by the values at the primes.  Unique factorisation nevertheless makes them orthogonal.  In the Steinhaus model,
\begin{equation}\label{eq:orthogonality}
        \E[f(m)\overline{f(n)}]=\ind_{m=n},
\end{equation}
whereas in the Rademacher model,
\begin{equation}\label{eq:rademacher-orthogonality}
        \E[f(m)f(n)]=\mu^2(n)\ind_{m=n}.
\end{equation}
Consequently,
$$
        \E[|M_f(x)|^2]=
        \begin{cases}
        \lfloor x\rfloor,&\text{in the Steinhaus model},\\
        \displaystyle\sum_{n\le x}\mu^2(n)\sim\dfrac{6}{\pi^2}x,
        &\text{in the Rademacher model},
        \end{cases}
$$
and $\sqrt{x}$ is the scale suggested by the second moment.

The second moment does not describe the typical size of $M_f(x)$.  Harper \cite{HarperLow} proved that, in both models,
\begin{equation}\label{eq:harper-first-moment}
        \E[|M_f(x)|]\asymp\frac{\sqrt{x}}{(\log\log x)^{1/4}},
\end{equation}
so that the partial sums exhibit better than squareroot cancellation.  He also determined the order of magnitude of $\E[|M_f(x)|^{2q}]$ for every $0\le q\le1$, and of the high moments \cite{HarperHigh}; see \cite{HarperReview} for a survey.  In Harper's proof these moments are controlled by the mean square of a random Euler product on the critical line.  After normalisation, the squared modulus of this Euler product behaves like a critical multiplicative chaos measure \cite{HarperLow}; see also \cite{SoundararajanZaman}.  Harper \cite{HarperCharacterZeta} has shown that the same phenomenon gives better than squareroot cancellation in typical Dirichlet character sums and zeta sums.

Gorodetsky and Wong \cite{GorodetskyWong} have since determined the limiting distribution of $M_f(x)$ in the Steinhaus model, answering a question of Harper.  They proved that, as $x\to\infty$,
$$
        \frac{(\log\log x)^{1/4}}{\sqrt{x}}M_f(x)\xrightarrow{\ d\ }\sqrt{V_\infty}\,G,
$$
where $G$ is a standard complex Gaussian random variable and $V_\infty$ is an almost surely positive and finite random variable, independent of $G$, given by the integral of $|\tfrac12+it|^{-2}$ against a critical multiplicative chaos measure.

Sums over short intervals behave differently.  Harper, Soundararajan and Xu \cite{HarperSoundararajanXu} showed, again in the Steinhaus model, that if $y\to\infty$ and $y=o(x)$, then $\sum_{x\le n\le x+y}f(n)$ has a Gaussian limiting distribution after division by a deterministic factor, and that this factor differs from $\sqrt{y}$ when $y$ is very close to $x$.

\subsection{Almost sure bounds}

The almost sure problem concerns the largest values of $M_f(x)$ as $x$ varies.  For comparison, if $(a_k)_{k\ge1}$ are independent symmetric signs, then the largest values of the partial sums $\sum_{k\le N}a_k$ exceed their typical size $\sqrt{N}$ by a factor of order $(\log\log N)^{1/2}$.  More precisely, the law of the iterated logarithm states that, almost surely,
$$
        \limsup_{N\to\infty}
        \frac{\left|\sum_{k\le N}a_k\right|}
        {\sqrt{2N\log\log N}}=1.
$$

Erd\H{o}s \cite[p.~251]{Erdos1961} asked whether the partial sums of a Rademacher random multiplicative function satisfy a law of the same form, namely whether there is a constant $c>0$ such that, almost surely,
$$
        \limsup_{N\to\infty}\frac{M_f(N)}{\sqrt{N\log\log N}}=c.
$$
This question is Problem 520 on Bloom's Erd\H{o}s Problems website \cite{ErdosProblems}.  In the Rademacher model, Theorem \ref{thm:main} with any $\eps<1/4$ gives
$$
        \lim_{N\to\infty}\frac{M_f(N)}{\sqrt{N\log\log N}}=0
$$
almost surely, so the answer to this question is negative.

For random multiplicative functions, \eqref{eq:harper-first-moment} shows that the typical size of $M_f(x)$ is smaller than $\sqrt{x}$.  If the passage from a fixed point to the largest values costs the same factor $(\log\log x)^{1/2}$ as for a random walk, then the largest values of $|M_f(x)|$ should have size about
$$
        \frac{\sqrt{x}}{(\log\log x)^{1/4}}
        (\log\log x)^{1/2}
        =\sqrt{x}(\log\log x)^{1/4}.
$$
This is only a heuristic, since the summands $f(n)$ are not independent and the same prime values determine $M_f(x)$ for every $x$.

Almost sure upper bounds for $M_f(x)$ go back to Wintner \cite{Wintner}, who proved in the Rademacher model that $M_f(x)\ll_{\eps,f}x^{1/2+\eps}$ almost surely, for every $\eps>0$.  Hal\'asz \cite{Halasz} improved this to
$$
        M_f(x)\ll_f\sqrt{x}\exp\left(C\sqrt{\log\log x\,\log\log\log x}\right)
$$
almost surely, for an absolute constant $C>0$.  Basquin \cite{Basquin} and, independently, Lau, Tenenbaum and Wu \cite{LauTenenbaumWu} reached the iterated-logarithmic scale, proving that
$$
        M_f(x)\ll_{\eps,f}\sqrt{x}(\log\log x)^{2+\eps}
$$
almost surely in the Rademacher model.  In the other direction, Harper's lower bound \eqref{eq:harper-lower} shows that the exponent $1/4$ cannot be replaced by any smaller fixed exponent in an almost sure upper bound of this form.

The first upper bound with the exponent $1/4$ was proved for a restricted sum.  Write $P(n)$ for the largest prime factor of $n$, with $P(1)=1$.  Mastrostefano \cite{Mastrostefano} proved in both models that, almost surely,
$$
        \left|\sum_{\substack{n\le x\\P(n)>\sqrt{x}}}f(n)\right|
        \ll_{\eps,f}\sqrt{x}(\log\log x)^{1/4+\eps}.
$$
An integer $n\le x$ with $P(n)>\sqrt{x}$ factors uniquely as $n=pm$ with $p=P(n)$ and $m<p$, so its largest prime factor occurs only once and can be separated from the rest of $n$.  Hardy \cite{HardyLargePrime} has since determined the limiting distribution of this restricted sum in the Steinhaus model.

In the full sum the largest prime factor can have any size, and Caich \cite{Caich} developed a decomposition by the largest prime factor that covers all of its possible sizes.  He proved that, in both models, almost surely
\begin{equation}\label{eq:caich-bound}
        |M_f(x)|\ll_{\eps,f}
        \sqrt{x}(\log\log x)^{3/4+\eps}.
\end{equation}

For the weighted sums $\sum_{n\le x}f(n)/\sqrt{n}$, which in the Steinhaus model serve as models for $\zeta(\tfrac12+it)$, sharp almost sure bounds in that model were obtained by Hardy \cite{HardyWeighted}, following work of Aymone, Heap and Zhao \cite{AymoneHeapZhao}.  In the Rademacher model, Atherfold \cite{Atherfold} obtained the analogue of \eqref{eq:caich-bound} for these sums.

\subsection*{Related work}

Harper's conjecture has attracted a great deal of attention, especially following the workshop on Probability in Number Theory held at the Centre de Recherches Math\'ematiques in Montr\'eal from 29 June to 3 July 2026, at which the problem of improving Caich's bound was raised repeatedly.  Several proofs, some of them of the Rademacher case only, were subsequently obtained around this time.  For completeness, we record the chronology of the results most closely related to ours.

A draft of our paper, containing Proposition \ref{prop:block-moment} in its present form, was sent to an independent expert on 10 July 2026, and was posted on the arXiv on 31 July 2026.

Korsky \cite{Korsky} independently resolved the Rademacher case of Theorem \ref{thm:main}.  His proof was listed as a proof claim on the Erd\H{o}s Problems forum on 29 July 2026, together with a partial formalisation in Lean carried out with V.~Reis.

H\o ystad \cite{Hoystad} subsequently announced, on 4 August 2026, a self-contained formalisation in Lean of the Rademacher case of Theorem \ref{thm:main}.

Verreault \cite{Verreault} posted his proof on the arXiv on 21 August 2026.  He proved Theorem~\ref{thm:main} independently of us and by a different route.  He also proved a refinement of the factor $(\log\log x)^{\eps}$ in Theorem \ref{thm:main}.  If $G\colon[3,\infty)\to(0,\infty)$ is eventually nondecreasing and
$$
        \sum_{h\ge3}\frac{h\log h}{G(h)^2}<\infty,
$$
then, in either model, almost surely
\begin{equation}\label{eq:gauge-bound}
        |M_f(x)|\ll_{f,G}\sqrt{x}(\log\log x)^{1/4}G(\log\log\log x).
\end{equation}
This implies Theorem \ref{thm:main}, by taking $G(h)=h^{1+\eps}$.  After Verreault emailed us, we revisited our proof of Theorem \ref{thm:main}, and found that it can be refined to prove \eqref{eq:gauge-bound} under the weaker hypothesis
$$
        \sum_{h\ge3}\frac{h}{G(h)^2}<\infty.
$$
For example, taking $G(h)=h(\log h)^{1/2}(\log\log h)^{1/2+\eps}$ gives, almost surely,
$$
        |M_f(x)|\ll_{f,\eps}\sqrt{x}(\log_2x)^{1/4}(\log_3x)(\log_4x)^{1/2}(\log_5x)^{1/2+\eps}
$$
for every $\eps>0$, where $\log_k$ denotes the $k$-fold iterated logarithm.  The corresponding bound in \cite[Corollary~1.2]{Verreault} has $\log_4x$ in place of $(\log_4x)^{1/2}$.  We thank him for raising this point with us.  Our proof of the refinement differs from his, and we do not give it here.

Gou and Xu \cite{GouXu} have since proved the analogue of \eqref{eq:harper-lower}, with fluctuations of either sign, for the extended Rademacher model, in which the values $f(p)=\pm1$ are extended completely multiplicatively.  This answers Erd\H{o}s Problem 1144 \cite{ErdosProblems} in the affirmative. 

\subsection*{Acknowledgements}
BD and APC thank the Heilbronn Institute for Mathematical Research for its support. We thank Christopher Atherfold, Rachid Caich and Adam Harper for a careful reading of a draft of this paper and for their helpful comments.  We also thank William Verreault and Samuel Korsky for their correspondence.

\section{Proof outline}\label{sec:outline}

We prove Theorem \ref{thm:main} in the following form, in which $r$ is the order of the conditional moments estimated in Proposition \ref{prop:block-moment}.

\begin{prop}\label{prop:fixed-r}
Let $f$ be either a Steinhaus random multiplicative function or a Rademacher random multiplicative function.  Let $r\ge2$ be an integer and let $\eta>0$.  Then, almost surely,
$$
        |M_f(x)|\ll_{r,\eta,f}
        \sqrt{x}(\log\log x)^{1/4+1/(2r)+\eta}.
$$
\end{prop}

Proposition \ref{prop:fixed-r} implies Theorem \ref{thm:main}, since for every $\eps>0$ we may choose an integer $r\ge2$ and $\eta>0$ such that $1/(2r)+\eta<\eps$.

Our proof keeps the architecture of Caich's argument \cite{Caich}.  We outline the proof of Proposition \ref{prop:fixed-r} in the Steinhaus model, and then describe the changes needed for the Rademacher model.  The problem is first reduced to a sparse sequence of test points $x_i$, between which the increments of $M_f(x)$ are negligible by Lemma \ref{lem:sparse-points} and Proposition \ref{prop:test-point-reduction}.  The test points are grouped into scales, where
$$
        x_i=\lfloor\exp(i^{c_0})\rfloor,\qquad X_{\ell-1}<x_i\le X_\ell,\qquad X_\ell=\exp(2^{\ell^K}).
$$
On the $\ell$th scale we put $L=\ell^K$, so that $L\asymp\log\log x_i$, and the primes above a level $y_0$ are divided into $J\asymp L$ blocks $(y_{j-1},y_j]$, each with reciprocal-prime mass $1/\ell+o(1/\ell)$.  The test points, scales and prime blocks are defined in Section \ref{sec:sparse}.

At a test point, the main part of $M_f(x_i)$ comes from the integers whose largest prime factor occurs only once.  Let $\Psi_f(z,y)$ and $\Psi_f'(z,y)$ denote the sums of $f(m)$ over the integers $m\le z$ with $P(m)\le y$ and with $P(m)<y$ respectively.  Lemma \ref{lem:decomposition} splits $M_f(x_i)$ into the sum over $y_0$-smooth integers, a sum $M_f^{(2)}(x_i)$ over integers whose largest prime factor is repeated, and the main term
$$
        M_f^{(1)}(x_i)=\sum_{y_0<p\le y_J}f(p)\Psi_f'\left(\frac{x_i}{p},p\right),
$$
in which these integers are grouped according to their largest prime factor $p$.  Lemma \ref{lem:B0-B2} disposes of the first two terms.  Since $\Psi_f'(x_i/p,p)$ involves only primes smaller than $p$, the main term is a martingale when the primes are revealed in increasing order, and its quadratic variation
$$
        V_\ell(x_i;f)=\sum_{y_0<p\le y_J}\left|\Psi_f'\left(\frac{x_i}{p},p\right)\right|^2
$$
has expectation of order $x_i$.

A martingale tail bound and a union bound over the test points reduce Theorem \ref{thm:main} to showing that, for every $\delta>0$, the quadratic variation is at most $x_i(\log\log x_i)^{-1/2+\delta}$ outside events of summable probability.  The stopped Hoeffding inequality (Lemma \ref{lem:stopped-hoeffding}) and a union bound over the $\exp(O(L))$ test points on a scale reduce the proof of Proposition \ref{prop:fixed-r} to showing that
$$
        \frac{V_\ell(x_i;f)}{x_i}\ll T(\ell)L^{-1/2+1/r},\qquad T(\ell)=\ell^{10},
$$
outside events of summable probability.  Since $T(\ell)=L^{10/K}$, this is the previous bound with $\delta=1/r+10/K$, and $\delta/2<1/(2r)+\eta$ by \eqref{eq:K-condition}.  This is a form of better than squareroot cancellation, and it must hold simultaneously at all test points.  The bound is Proposition \ref{prop:variance}, and the final step is carried out in Section \ref{sec:fixedr}.

Caich's smoothing inequality (Proposition \ref{prop:caich-smoothing}) bounds $V_\ell(x_i;f)/x_i$ by auxiliary terms and $\ell\log\ell$ times the largest of the principal mean squares
$$
        M_\ell(x_i,y_j;f)=\frac1{\log y_j}\int_{x_i/y_j}^{x_i/y_{j-1}}\left|\Psi_f\left(z,\frac{x_i}{z}\right)\right|^2\frac{\dee z}{z^2}\qquad(1\le j\le J).
$$
These are logarithmic mean squares of smooth partial sums whose smoothness cutoff $x_i/z$ lies in the $j$th block.  Outside events of summable probability, the auxiliary terms are $O(T(\ell)L^{-1/2})$ by Lemma \ref{lem:lambda-two} and Proposition \ref{prop:caich-smoothing}.  For every test point on the scale, $M_\ell(x_i,y_j;f)$ is at most the block variable
$$
        U_j=\frac1{\log y_j}\int_0^\infty\max_{q\in\CQ_j}|\Psi_f(z,q)|^2\frac{\dee z}{z^2},
$$
a mean square maximised over the smoothness cutoffs $q\in\CQ_j$, which are $y_{j-1}$ and the primes in the $j$th block.

By Parseval's identity, the mean square on the $j$th block is controlled by a normalised mean square $I_{j-1}$, over the critical line, of the random Euler product formed from the primes up to $y_{j-1}$.  If $F_y(s)=\prod_{p\le y}(1-f(p)p^{-s})^{-1}$, then Lemma \ref{lem:parseval-dirichlet} gives
$$
        \frac1{\log y}\int_0^\infty|\Psi_f(z,y)|^2\frac{\dee z}{z^2}=\frac1{2\pi\log y}\int_{-\infty}^{\infty}\left|\frac{F_y(\tfrac12+it)}{\tfrac12+it}\right|^2\dee t.
$$
The variable $I_j$ in \eqref{eq:Ij} is $2\pi$ times the right-hand side with $y=y_j$, multiplied by the tilt $(\log y_j/\log y_0)^{-1/L}$, which changes it by only a bounded factor over all $J$ blocks.

Harper's low-moment estimates show that $I_{j-1}$ is typically at most of order $L^{-1/2}$.  In the Steinhaus model, the corresponding mean square for the primes up to $y$, multiplied by $(\log\log y)^{1/2}$, converges in probability as $y\to\infty$ to a constant multiple of the random variable $V_\infty$ of Section \ref{sec:typical}, as follows from \cite[Theorem~1.1 and Corollary~1.2]{GorodetskyWong}.

Caich controls the mean square on each block through its conditional first moment, and the union bound over the $J$ blocks then loses a factor $L$.  This loss is the source of the exponent $3/4$ in \eqref{eq:caich-bound}.  The new ingredient in our proof is Proposition \ref{prop:block-moment}, which replaces this step by a bound for a conditional moment of fixed order $r\ge2$ on a single block, namely
$$
        \E\left[U_j^r\mid\CF_{y_{j-1}}\right]\ll_rI_{j-1}^r.
$$
Here $\CF_y$ is the $\sigma$-algebra generated by the values $f(p)$ with $p\le y$, so the bound states that the $r$th moment of the mean square on the $j$th block, maximised over the smoothness cutoffs in that block and conditioned on the values $f(p)$ at the smaller primes, is $O_r(I_{j-1}^r)$.

The proof of Proposition \ref{prop:block-moment} reveals the primes of one block in increasing order, removes the maximum over the cutoffs by Doob's inequality, applies the hypercontractive inequality to the new primes, and finishes with Parseval's identity.  The blocks are thin, in the sense that the sum of $1/p$ over each of them tends to zero, and so the divisor weights produced by the hypercontractive inequality on a single block contribute only the factor
$$
        \prod_{y_{j-1}<p\le y_j}\left(1-\frac1p\right)^{-(2r-1)}=1+O_r\left(\frac1\ell\right).
$$

A stopping argument keeps the realised values $I_{j-1}$ below $L^{-1/2}$, up to a slowly growing factor.  The tilt in $I_j$ makes $(I_j)$ a supermartingale (Lemma \ref{lem:supermartingale}), and Harper's low-moment estimate (Lemma \ref{lem:steinhaus-harper}) then shows that $\max_jI_j\le T_1(\ell)^{1/2}L^{-1/2}$ outside an event of probability $O(T_1(\ell)^{-1/3})$, where $T_1(\ell)=T(\ell)/(\ell\log\ell)$.  On this event, the conditional Markov inequality bounds the probability that $U_j$ exceeds a constant multiple of $T_1(\ell)L^{-1/2+1/r}$ by $T_1(\ell)^{-r/2}L^{-1}$ for each block.  This is smaller than $1/J$ by a power of $\ell$, so the union bound over the $J\ll L$ blocks still leaves events of summable probability, which gives Proposition \ref{prop:principal}.  Multiplying by the factor $\ell\log\ell=T(\ell)/T_1(\ell)$ in the smoothing inequality then gives Proposition \ref{prop:variance}.

Apart from Proposition \ref{prop:block-moment}, the stopping argument in Proposition \ref{prop:principal} and the Rademacher low-moment estimate below, the structure of the proof is that of Caich \cite{Caich}.

In the Rademacher model (Section \ref{sec:rademacher}), the squarefree support removes $M_f^{(2)}(x_i)$, and the Euler factors become $1+1/p$ in the supermartingale and $1+(2r-1)/p$ in the block moment.  The smoothing estimates, the supermartingale property and the block moment carry over (Propositions \ref{prop:rademacher-smoothing} and \ref{prop:rademacher-block-moment} and Lemma \ref{lem:rademacher-supermartingale}).  The one new input is the low-moment estimate of Lemma \ref{lem:rademacher-low-moment}.  Rademacher signs are not invariant under $f(p)\mapsto f(p)p^{-iN}$, so the translation argument in the proof of Lemma \ref{lem:steinhaus-harper} is not available.  We use instead Harper's estimate on shifted unit intervals (Lemma \ref{lem:harper-shifted-rademacher}), summed against the decay of $|\tfrac12+it|^{-2}$.

\section{Preliminaries}\label{sec:prelim}

Throughout Sections \ref{sec:prelim}--\ref{sec:fixedr}, $f$ denotes a Steinhaus random multiplicative function.  All logarithms are natural.  Sums over $n\le x$ are understood to be empty when $x<1$.  For $y\ge2$, let $\CF_y$ be the $\sigma$-algebra generated by $\{f(p):p\le y\}$.

We write $P(n)$ for the largest prime factor of $n$, with the convention $P(1)=1$.  For a prime $p$, let $v_p(n)=\max\{a\ge0:p^a\mid n\}$, and let $\omega(n)$ denote the number of distinct prime factors of $n$.  Finally, $\tau_k(n)$ denotes the number of ordered $k$-tuples of positive integers with product $n$, and we use the Chebyshev functions
$$
        \psi(x)=\sum_{n\le x}\Lambda(n),
        \qquad
        \vartheta(x)=\sum_{p\le x}\log p.
$$

Fix an integer $r\ge2$ and $\eta>0$.  Choose an integer $K\ge100$ such that
\begin{equation}\label{eq:K-condition}
        2K\eta>10.
\end{equation}
A sparseness parameter $0<c_0\le10^{-3}$ will be chosen in Lemma \ref{lem:sparse-points}, and a fixed integer $m_0>1/c_0$ will be used in the smoothing estimate.  The growing moment used only for the repeated-prime contribution is $m=L$, where $L=\ell^K$ is defined in Section \ref{sec:sparse}.

\subsection{Probabilistic tools}\label{sec:prob-tools}
We begin with the moment inequality that replaces orthogonality when moments higher than the second are needed.  For Steinhaus random multiplicative functions it bounds the $2m$-th moment of a linear combination $\sum_n a_nf(n)$ by the $m$-th power of a weighted $\ell^2$-norm of the coefficients, with weights given by divisor functions.  The underlying product-space inequality goes back to Bonami \cite{Bonami}; the multiplicative form below is stated by Caich \cite{Caich}.  It will be used both directly, for the increments in Lemma \ref{lem:sparse-points}, and conditionally, in Proposition \ref{prop:block-moment}, on the product space of one block of new prime coordinates.

\begin{lem}\label{lem:hypercontractive}
Let $f$ be a Steinhaus random multiplicative function.  For every sequence $(a_n)$ of complex numbers and every positive integer $m$,
$$
        \E\left[\left|\sum_n a_nf(n)\right|^{2m}\right]
        \le
        \left(\sum_n|a_n|^2\tau_{2m-1}(n)\right)^m,
$$
whenever the right-hand side is finite.
\end{lem}

\begin{proof}
This is Caich \cite[Lemma 3.1]{Caich}, obtained from the product-space hypercontractive inequality of Bonami \cite{Bonami}.
\end{proof}
We shall use Parseval's identity for Dirichlet series in the following form.
\begin{lem}\label{lem:parseval-dirichlet}
Let
$$
        A(s)=\sum_{n\ge1}\frac{a_n}{n^s}
$$
be a Dirichlet series with abscissa of convergence $\sigma_c$.  For every $\sigma>\max(0,\sigma_c)$,
$$
        \int_0^\infty
        \left|\sum_{n\le x}a_n\right|^2\frac{\dee x}{x^{1+2\sigma}}
        =
        \frac1{2\pi}\int_{-\infty}^{\infty}
        \left|\frac{A(\sigma+it)}{\sigma+it}\right|^2\dee t.
$$
\end{lem}

\begin{proof}
See Montgomery--Vaughan \cite[(5.26)]{MV}.
\end{proof}

We use both of Doob's maximal inequalities.

\begin{lem}\label{lem:doob-inequalities}
Let $(X_j)_{0\le j\le J}$ be a nonnegative submartingale.  If $p>1$, then
$$
        \E\left[\max_{0\le j\le J}X_j^p\right]
        \le\left(\frac{p}{p-1}\right)^p\E[X_J^p].
$$
If $(Y_j)_{0\le j\le J}$ is a nonnegative supermartingale, then, for every $u>0$,
$$
        \P\left(\max_{0\le j\le J}Y_j>u\right)
        \le\frac{\E[Y_0]}{u}.
$$
\end{lem}

\begin{proof}
See Caich \cite[Lemmas 3.8 and 3.9]{Caich}.
\end{proof}
Finally we will use the following variant of the Azuma--Hoeffding  inequality.

\begin{lem}\label{lem:stopped-hoeffding}
Let $(Z_n)_{1\le n\le N}$ be a complex martingale difference sequence with respect to a filtration $(\CG_n)_{0\le n\le N}$.  Suppose that the $Z_n$ are bounded and $|Z_n|\le S_n$, where each $S_n$ is nonnegative and $\CG_{n-1}$-measurable.  For $T>0$, let
$$
        \Sigma(T)=\left\{\sum_{n\le N}S_n^2\le T\right\}.
$$
Then, for every $u>0$,
$$
        \P\left(
        \left\{\left|\sum_{n\le N}Z_n\right|\ge u\right\}
        \cap\Sigma(T)\right)
        \le2\exp\left(-\frac{u^2}{10T}\right).
$$
\end{lem}

\begin{proof}
See Caich \cite[Lemma 3.11]{Caich}.  The proof applies the exponential-moment argument separately to the real and imaginary parts and stops when the accumulated predictable bound first exceeds $T$.
\end{proof}

\subsection{Number-theoretic estimates}\label{sec:nt-tools}
We need the following consequence of Shiu's theorem for divisor sums in short intervals, uniform for intervals of length at least $x^\beta$.
\begin{lem}\label{lem:shiu-divisor}
Let $k\ge1$ be fixed and let $0<\beta<1$.  Uniformly for large $x$ and $x^\beta\le h\le x$,
$$
        \sum_{x<n\le x+h}\tau_k(n)
        \ll_{k,\beta}h(\log x)^{k-1}.
$$
\end{lem}

\begin{proof}
Apply Shiu's Brun--Titchmarsh theorem for multiplicative functions \cite{Shiu} with modulus $1$ and $g=\tau_k$.  Its growth hypotheses follow from the estimates
$$
        \tau_k(p^\nu)=\binom{\nu+k-1}{k-1}\le k^\nu,
        \qquad
        \tau_k(n)\ll_{k,\delta}n^\delta
$$
for every $\delta>0$.  Shiu's estimate and Mertens' theorem give
$$
\begin{aligned}
        \sum_{x<n\le x+h}\tau_k(n)
        &\ll_{k,\beta}\frac{h}{\log(x+h)}
        \exp\left(k\sum_{p\le x+h}\frac1p\right)\\
        &\ll_{k,\beta}h(\log x)^{k-1}.
\end{aligned}
$$
The constants are uniform throughout the stated interval range.
\end{proof}

We shall use Mertens' theorem with an error smaller than any power of $1/\log u$.  In Lemma \ref{lem:supermartingale} the required error is $o(1/(\ell L))$, while Lemma \ref{lem:short-reciprocal-primes} applies the same estimate on the intervals $(v/(1+1/X),v]$.

\begin{lem}\label{lem:quantitative-mertens}
There are absolute constants $B_1\in\R$ and $c>0$ such that, uniformly for $u\ge3$,
\begin{equation}\label{eq:quantitative-mertens}
        \sum_{p\le u}\frac1p
        =\log\log u+B_1+
        O\left(\exp(-c\sqrt{\log u})\right).
\end{equation}
\end{lem}

\begin{proof}
Montgomery--Vaughan \cite[Theorem 6.9]{MV} give
$$
        \psi(u)=u+O\left(u\exp(-c_1\sqrt{\log u})\right).
$$
Removing the prime-power contribution gives the same form for $\vartheta(u)$, and partial summation yields \eqref{eq:quantitative-mertens}.
\end{proof}

The smooth-number estimate needed below is given with uniform dependence.

\begin{lem}\label{lem:smooth-number-rankin}
For $z\ge y\ge3$, let
$$
        \Psi(z,y)=\#\{n\le z:P(n)\le y\},
        \qquad
        u=\frac{\log z}{\log y}.
$$
There is an absolute constant $C>0$ such that, whenever $u\ge3$ and $\log u\le\frac12\log y$,
\begin{equation}\label{eq:smooth-number-rankin}
        \Psi(z,y)
        \le z\exp\left(
        -\frac12u\log u
        +C\log\log y
        +C\frac{\sqrt{u}}{\log(2u)}
        \right).
\end{equation}
Consequently, if $u=\log\log z$ and $y=z^{1/u}$, then
\begin{equation}\label{eq:smooth-number-consequence}
        \Psi(z,y)
        \ll z(\log z)^{-c\log\log\log z}
\end{equation}
for some absolute $c>0$.
\end{lem}

\begin{proof}
Let $\delta=(\log u)/(2\log y)$ and $\sigma=1-\delta$, so that $\sigma\ge3/4$ by the hypotheses.  Rankin's inequality gives
$$
        \Psi(z,y)
        \le z^\sigma\prod_{p\le y}(1-p^{-\sigma})^{-1}
        =z\exp\left(-\frac12u\log u\right)\prod_{p\le y}(1-p^{-\sigma})^{-1},
$$
because $\delta\log z=\frac12u\log u$.  As $\sigma\ge3/4$,
$$
        \log\prod_{p\le y}(1-p^{-\sigma})^{-1}
        \le\sum_{p\le y}p^{-1+\delta}+O(1)
        \ll\log\log y+\frac{y^\delta}{\delta\log y},
$$
uniformly for $0<\delta\le1/4$, where the second bound follows by partial summation from the elementary estimate $\pi(t)\ll t/\log t$.  Here $y^\delta=\sqrt{u}$ and $\delta\log y=\frac12\log u$, which proves \eqref{eq:smooth-number-rankin}.  Since
$$
        \log\log y=\log\log z+O(\log\log\log z)
$$
in the final assertion, the negative term $-\frac12u\log u$ dominates the two error terms and gives \eqref{eq:smooth-number-consequence}.
\end{proof}

\subsection{Test points and prime blocks}\label{sec:sparse}

We now set up the construction of Caich \cite{Caich}, in the notation fixed at the beginning of Section \ref{sec:prelim}.  Recall that $0<c_0\le10^{-3}$ and that $K\ge100$ is an integer. We define
$$
        x_i=\left\lfloor\exp(i^{c_0})\right\rfloor,
        \qquad
        X_\ell=\exp\left(2^{\ell^K}\right),
        \qquad
        L=\ell^K.
$$
We call the numbers $x_i$ the test points.  For each index $i$ satisfying $X_{\ell-1}<x_i\le X_\ell$, define
$$
        y_j=\exp\left(e^{j/\ell}
        (\log X_\ell)^{1-K/\ell}\right)
        \qquad(j\ge0).
$$
Let $J$ be minimal with $y_J\ge X_\ell$.  Then
\begin{equation}\label{eq:J-count}
        J=\left\lceil K(\log2)L\right\rceil\ll_KL,
\end{equation}
and, uniformly for indices $i$ satisfying $X_{\ell-1}<x_i\le X_\ell$ and for $0\le j\le J$,
\begin{equation}\label{eq:L-comparability}
        L\asymp\log\log x_i\asymp\log\log y_j.
\end{equation}
Moreover,
\begin{equation}\label{eq:block-ratio}
        \frac{\log y_j}{\log y_{j-1}}=e^{1/\ell}
        \qquad(1\le j\le J).
\end{equation}

\section{Reduction to the principal mean square}\label{sec:reduction}

In this section we reduce the Steinhaus case of Proposition \ref{prop:fixed-r} to a bound for the principal mean squares $M_\ell(x_i,y_j;f)$ defined in Section \ref{sec:smoothing}.  The increments between test points, the smooth and repeated-prime terms, and the auxiliary terms in the smoothing inequality are all shown to be negligible outside events of summable probability.

\subsection{Interpolation between test points}\label{sec:interpolation}

The first lemma of this section shows that the fluctuation of $M_f(x)$ between consecutive test points is, almost surely, smaller than $\sqrt{x_i}$ by any fixed power of $\log x_i$; consequently it will suffice to bound $M_f(x_i)$, as Proposition \ref{prop:test-point-reduction} shows.  The proof combines the hypercontractive inequality (Lemma \ref{lem:hypercontractive}), a chaining decomposition of the interval $(x_{i-1},x_i]$ into dyadic blocks, and the short-interval divisor bound (Lemma \ref{lem:shiu-divisor}); it is at this point that the sparseness parameter $c_0$ is fixed.

\begin{lem}\label{lem:sparse-points}
For every fixed $A>0$, one can choose $0<c_0=c_0(A)\le10^{-3}$ such that, almost surely, for all sufficiently large $i$,
$$
        \max_{x_{i-1}<x\le x_i}
        |M_f(x)-M_f(x_{i-1})|
        \ll_A\frac{\sqrt{x_i}}{(\log x_i)^A}.
$$
\end{lem}

\begin{proof}[Proof of Lemma \ref{lem:sparse-points}]
Let $h_i=x_i-x_{i-1}$.  The mean value theorem gives
$$
        h_i\asymp\frac{x_i}{i^{1-c_0}}.
$$
Fix an integer $m>2$.  For an interval $I\subset(x_{i-1},x_i]$, Lemma \ref{lem:hypercontractive} gives
\begin{equation}\label{eq:interval-moment}
        \E\left[\left|\sum_{n\in I}f(n)\right|^{2m}\right]
        \le\left(\sum_{n\in I}\tau_{2m-1}(n)\right)^m.
\end{equation}
For each dyadic $\Delta=2^\nu\le h_i$, partition $(x_{i-1},x_i]$ into consecutive intervals $I_{\nu,b}$ of length $\Delta$, with one possibly shorter terminal interval.  Every initial subinterval is a disjoint union of at most two such intervals at each scale.  Minkowski's inequality and $\E[\max_b|Y_b|^{2m}]\le\sum_b\E[|Y_b|^{2m}]$ therefore give
\begin{equation}\label{eq:chaining}
\begin{aligned}
&\left\|
\max_{x_{i-1}<x\le x_i}|M_f(x)-M_f(x_{i-1})|
\right\|_{2m}\ll
\sum_{0\le\nu\le\lfloor\log h_i/\log2\rfloor}
\left(
\sum_b\E\left[\left|\sum_{n\in I_{\nu,b}}f(n)\right|^{2m}\right]
\right)^{1/(2m)}.
\end{aligned}
\end{equation}

Let $k=2m-1$.  H\"older's inequality gives
$$
        \left(\sum_{n\in I_{\nu,b}}\tau_k(n)\right)^m
        \le\Delta^{m-1}\sum_{n\in I_{\nu,b}}\tau_k(n)^m.
$$
For fixed $m$, there is an integer $K_m$ such that $\tau_k(n)^m\le\tau_{K_m}(n)$ for every $n$.  Summing over $b$, using \eqref{eq:interval-moment}, and applying Lemma \ref{lem:shiu-divisor} to the entire interval $(x_{i-1},x_i]$, we obtain
$$
        \sum_b\E\left[\left|\sum_{n\in I_{\nu,b}}f(n)\right|^{2m}\right]
        \ll_m\Delta^{m-1}h_i(\log x_i)^{C_m}.
$$
The application of Shiu's lemma is uniform because
$$
        h_i\asymp
        \frac{x_i}{(\log x_i)^{(1-c_0)/c_0+o(1)}}
        \ge x_i^\beta
$$
for every fixed $\beta<1$ once $i$ is large.  Substitution in \eqref{eq:chaining} and summation of the geometric series therefore give
$$
        \E\left[
        \max_{x_{i-1}<x\le x_i}
        |M_f(x)-M_f(x_{i-1})|^{2m}
        \right]
        \ll_mh_i^m(\log x_i)^{C_m}.
$$
Consequently,
$$
\begin{aligned}
&\P\left(
\max_{x_{i-1}<x\le x_i}|M_f(x)-M_f(x_{i-1})|
>\frac{\sqrt{x_i}}{(\log x_i)^A}
\right)\\
&\qquad\ll_{m,A}
i^{-m(1-c_0)}(\log x_i)^{C_m+2Am}
\ll_{m,A}
i^{-m(1-c_0)+c_0(C_m+2Am)}.
\end{aligned}
$$
Choose $m>2$ and then $c_0>0$ so small that the final exponent is less than $-1$.  Borel--Cantelli proves the lemma.
\end{proof}

The next proposition reduces the problem to the test points.

\begin{prop}\label{prop:test-point-reduction}
Let $R(x)$ be a fixed power of $\log\log x$, and suppose that
\begin{equation}\label{eq:reduction-goal}
        \sum_{\ell\ge1}\P\left(
        \sup_{X_{\ell-1}<x_i\le X_\ell}
        \frac{|M_f(x_i)|}{\sqrt{x_i}R(x_i)}>3\right)<\infty.
\end{equation}
Then, almost surely, $|M_f(x)|\ll\sqrt{x}R(x)$ for all sufficiently large $x$.
\end{prop}

\begin{proof}
By the first Borel--Cantelli lemma, almost surely $|M_f(x_i)|\le3\sqrt{x_i}R(x_i)$ for all sufficiently large $i$.  If $x_{i-1}<x\le x_i$, then by the triangle inequality
\begin{equation*}
        |M_f(x)|
        \le|M_f(x_i)|
        +2\max_{x_{i-1}<t\le x_i}|M_f(t)-M_f(x_{i-1})|.
\end{equation*}
By Lemma \ref{lem:sparse-points}, the second term is
$$
        O_A\left(\sqrt{x_i}(\log x_i)^{-A}\right)=o\left(\sqrt{x_i}R(x_i)\right).
$$
Since $x_i/x_{i-1}\to1$, we have $\sqrt{x_i}R(x_i)\asymp\sqrt{x}R(x)$ uniformly for $x_{i-1}<x\le x_i$, and the proposition follows.
\end{proof}

We shall apply Proposition \ref{prop:test-point-reduction} with $R(x)=(\log\log x)^{1/4+1/(2r)+\eta}$; the factor $3$ accommodates the three terms in \eqref{eq:largest-prime}.

\subsection{The largest prime factor decomposition}\label{sec:decomposition}

For $y\ge2$, set
$$
        \Psi_f(x,y)=\sum_{\substack{n\le x\\P(n)\le y}}f(n),
        \qquad
        \Psi_f'(x,y)=\sum_{\substack{n\le x\\P(n)<y}}f(n),
$$
the partial sums of $f$ over $y$-smooth integers, with the weak and strict smoothness conditions respectively. At each test point $x_i$ with $X_{\ell-1}<x_i\le X_\ell$, define
$$
        M_f^{(1)}(x_i)
        =\sum_{y_0<p\le y_J}
        f(p)\Psi_f'\left(\frac{x_i}{p},p\right)
$$
and
$$
        M_f^{(2)}(x_i)
        =
        \sum_{\substack{n\le x_i,\ P(n)>y_0\\
        v_{P(n)}(n)\ge2}}f(n).
$$

The next lemma is the combinatorial decomposition of $M_f(x_i)$ described in Section \ref{sec:outline}, together with a block-by-block expansion of the repeated-prime part which prepares it for the conditional moment estimates of the next subsection.

\begin{lem}\label{lem:decomposition}
For $X_{\ell-1}<x_i\le X_\ell$,
\begin{equation}\label{eq:largest-prime}
        M_f(x_i)=
        \Psi_f(x_i,y_0)+M_f^{(1)}(x_i)+M_f^{(2)}(x_i).
\end{equation}
Moreover,
\begin{equation}\label{eq:largest-prime-multiple}
 M_f^{(2)}(x_i)
 =\sum_{j=1}^{J}
 \sum_{\substack{y_{j-1}^{2}<d\le x_i\\
 y_{j-1}<q\le y_j\ {\rm for\ every\ prime}\ q\mid d\\
 v_{P(d)}(d)\ge2}}
 f(d)\Psi_f\left(\frac{x_i}{d},y_{j-1}\right).
\end{equation}
\end{lem}

\begin{proof}
Since $y_J\ge X_\ell\ge x_i$, every integer $n\le x_i$ either satisfies $P(n)\le y_0$ or has its largest prime factor $P(n)$ in a unique half-open block $(y_{j-1},y_j]$ with $1\le j\le J$.  The integers of the first class contribute $\Psi_f(x_i,y_0)$.

Let $p=P(n)>y_0$.  If $v_p(n)=1$, write $n=pm$.  Then $P(m)<p$; the inequality is strict even when $m$ has another prime factor in the same block, because $p$ divides $n$ exactly once.  By complete multiplicativity $f(n)=f(p)f(m)$, so $n$ occurs exactly once in $f(p)\Psi_f'(x_i/p,p)$; conversely every term of $M_f^{(1)}(x_i)$ arises in this way from a unique such $n$.

Suppose instead that $v_p(n)\ge2$, and let $(y_{j-1},y_j]$ be the unique block containing $p$.  Set
\begin{equation*}
        d=\prod_{\substack{q\mid n,\ q\ \mathrm{prime}\\
        y_{j-1}<q\le y_j}}q^{v_q(n)},
        \qquad m=\frac{n}{d}.
\end{equation*}
Then $P(d)=p$ and $v_{P(d)}(d)=v_p(n)\ge2$, so $d\ge p^{2}>y_{j-1}^{2}$.  Every prime factor of $d$ lies in $(y_{j-1},y_j]$, and $P(m)\le y_{j-1}$, since any prime factor of $m$ exceeding $y_{j-1}$ would lie in the block $(y_{j-1},y_j]$ (it cannot exceed $y_j\ge p=P(n)$) and hence would divide $d$.

Conversely, a term $dm$ on the right-hand side of \eqref{eq:largest-prime-multiple} corresponds to the integer $n=dm\le x_i$, whose largest prime factor $P(d)>y_0$ has multiplicity $v_{P(d)}(d)\ge2$.  Unique factorisation and the uniqueness of the half-open block recover $d$ and $m$ from $n$.  Hence \eqref{eq:largest-prime-multiple} counts each remaining integer exactly once, with the correct weight $f(n)=f(d)f(m)$, proving both identities.  This is the decomposition used by Caich \cite{Caich}.
\end{proof}

Given a positive function $R$, we now define the three families of bad events whose probabilities must be summed.  For $a=0,1,2$, let
\begin{equation}\label{eq:Bcal-def}
        \CB_\ell^{(a)}
        =\left\{\sup_{X_{\ell-1}<x_i\le X_\ell}
        \frac{|M_f^{(a)}(x_i)|}{\sqrt{x_i}R(x_i)}>1\right\},
\end{equation}
with the convention $M_f^{(0)}(x_i)=\Psi_f(x_i,y_0)$.  In view of \eqref{eq:largest-prime} and the reduction \eqref{eq:reduction-goal}, Proposition \ref{prop:fixed-r} will follow once we show that
\begin{equation*}
        \sum_{\ell\ge1}\P(\CB_\ell^{(a)})<\infty
        \qquad(a=0,1,2),
\end{equation*}
with $R(x)=(\log\log x)^{1/4+1/(2r)+\eta}$.

\subsection{Smooth and repeated-prime contributions}\label{sec:repeated}
This subsection disposes of the events $\CB_\ell^{(0)}$ and $\CB_\ell^{(2)}$.  Recall from the beginning of Section \ref{sec:prelim} that in this subsection the moment parameter is $m=L=\ell^{K}$.  For $X_{\ell-1}<x_i\le X_\ell$ and $1\le j\le J$, define
$$
        N_{i,j}=
        \sum_{\substack{y_{j-1}^{2}<d\le x_i\\
        v_{P(d)}(d)\ge2\\
        p\mid d\Rightarrow y_{j-1}<p\le y_j}}
        f(d)\Psi_f\left(\frac{x_i}{d},y_{j-1}\right)
$$
so that $M_f^{(2)}(x_i)=\sum_{j=1}^{J}N_{i,j}$ by \eqref{eq:largest-prime-multiple}, and let
$$
        X_{i,j}=
        \sum_{\substack{y_{j-1}^{2}<d\le x_i\\
        v_{P(d)}(d)\ge2\\
        p\mid d\Rightarrow y_{j-1}<p\le y_j}}
        \tau_{2m-1}(d)
        \left|\Psi_f\left(\frac{x_i}{d},y_{j-1}\right)\right|^2.
$$
Note that $X_{i,j}$ is $\CF_{y_{j-1}}$-measurable. Let
\begin{equation}\label{eq:Xcal-def}
        \CX_\ell=
        \left\{
        \sup_{\substack{X_{\ell-1}<x_i\le X_\ell\\1\le j\le J}}
        \frac{X_{i,j}}{x_i}\le\ell^{-10K}
        \right\}.
\end{equation}

The next lemma collects the conditional moment bound for $N_{i,j}$ (an application of the hypercontractive inequality on the block, parallel to the one we shall make in Proposition \ref{prop:block-moment}) together with a first-moment bound for $X_{i,j}$, which is small because integers $d>y_{j-1}^{2}$ composed of block primes with a repeated largest prime factor are rare.

\begin{lem}\label{lem:Nij}
For all sufficiently large $\ell$ in terms of $K$, uniformly in $i$ and $1\le j\le J$,
$$
        \E[|N_{i,j}|^{2m}\mid\CF_{y_{j-1}}]\le X_{i,j}^m
$$
and
$$
        \E[X_{i,j}]
        \ll_{K,c_0}
        \frac{x_im^2e^{Cm/\ell}}{y_{j-1}},
$$
where $C>0$ is absolute.
\end{lem}

\begin{proof}
The conditional inequality is the block hypercontractive calculation in Caich \cite[equation (13)]{Caich}.  There the coefficients are the smooth sums over the smaller primes, which are fixed by the conditioning, and the new prime variables are those indexed by the primes in $(y_{j-1},y_j]$.  Caich's calculation leading to \cite[equation (15)]{Caich} then gives the stated first moment.

That calculation requires every block prime to exceed $2m-1$.  Here
$$
        \log y_0=2^{\ell^K-K\ell^{K-1}},
        \qquad m=\ell^K,
$$
so the requirement holds for all sufficiently large $\ell$.  The final block causes no range change.  Because $d\le x_i$, primes exceeding $x_i$ contribute nothing, and truncating the last block at $x_i$ leaves both calculations unchanged.
\end{proof}

The following lemma is a variant of Caich's argument, stated for a general positive function $R$, so we include the proof.  Its first assertion shows that the event $\CX_\ell$ fails with summable probability, and its second bounds the probability of $\CB_\ell^{(2)}$ on $\CX_\ell$, where the conditional bound of order $2m$ is available.
\begin{lem}\label{lem:repeated-prime-probability}
One has
$$
        \sum_{\ell\ge1}\P(\bar\CX_\ell)<\infty.
$$
For every positive function $R$, writing
$$
        R_\ell=\inf_{X_{\ell-1}<x_i\le X_\ell}R(x_i),
$$
one also has, for all sufficiently large $\ell$,
\begin{equation}\label{eq:repeated-general-R}
        \P(\CB_\ell^{(2)}\cap\CX_\ell)
        \ll_{K,c_0}
        2^{L/c_0}J
        \left(\frac{J^2}{\ell^{10K}R_\ell^2}\right)^m.
\end{equation}
\end{lem}

\begin{proof}
There are $O(2^{L/c_0})$ indices $i$ satisfying $X_{\ell-1}<x_i\le X_\ell$.  Markov's inequality, Lemma \ref{lem:Nij}, $J\ll_KL$, and $y_{j-1}\ge y_0$ give
$$
\begin{aligned}
        \P(\bar\CX_\ell)
        &\le
        \sum_{\substack{X_{\ell-1}<x_i\le X_\ell\\1\le j\le J}}
        \frac{\ell^{10K}\E[X_{i,j}]}{x_i}\\
        &\ll_{K,c_0}
        2^{L/c_0}J\ell^{10K}
        \frac{m^2e^{Cm/\ell}}{y_0}.
\end{aligned}
$$
Since $y_0=\exp(2^{L-K\ell^{K-1}})$, this is summable in $\ell$.

For \eqref{eq:repeated-general-R}, if $|\sum_{j=1}^JN_{i,j}|>\sqrt{x_i}R(x_i)$, then $|N_{i,j}|>\sqrt{x_i}R(x_i)/J$ for some $j$.  The event $\{X_{i,j}\le x_i\ell^{-10K}\}$ contains $\CX_\ell$ and is $\CF_{y_{j-1}}$-measurable, so the conditional Markov inequality and Lemma \ref{lem:Nij} give
$$
\begin{aligned}
&\P\left(
|N_{i,j}|>\frac{\sqrt{x_i}R(x_i)}J,\
X_{i,j}\le x_i\ell^{-10K}
\right)\\
&\qquad\le
\left(\frac{J^2}{x_iR(x_i)^2}\right)^m
\E\left[
\ind_{\{X_{i,j}\le x_i\ell^{-10K}\}}X_{i,j}^m
\right]\\
&\qquad\le
\left(\frac{J^2}{\ell^{10K}R(x_i)^2}\right)^m.
\end{aligned}
$$
The union over $i$ and $j$ proves \eqref{eq:repeated-general-R}.
\end{proof}

The next lemma disposes of the smooth and repeated-prime contributions.

\begin{lem}\label{lem:B0-B2}
Let $R(x)=(\log\log x)^\alpha$.  Then
$$
        \sum_{\ell\ge1}\P(\CB_\ell^{(0)})<\infty
$$
for every fixed $\alpha\in\R$, and
$$
        \sum_{\ell\ge1}\P(\CB_\ell^{(2)})<\infty
$$
for every fixed $\alpha>1/4$.
\end{lem}

\begin{proof}
For $X_{\ell-1}<x_i\le X_\ell$, let
$$
        u_i=\frac{\log x_i}{\log y_0}.
$$
Since $x_i>X_{\ell-1}$ and $\log y_0=2^{\ell^K-K\ell^{K-1}}$, the binomial theorem gives the base-two calculation
$$
\begin{aligned}
        \frac{\log u_i}{\log 2}
        &\ge(\ell-1)^K-\ell^K+K\ell^{K-1}\\
        &=\frac{K(K-1)}2\ell^{K-2}+O_K(\ell^{K-3}).
\end{aligned}
$$
On the other hand, $x_i\le X_\ell$ gives
$$
        \log\log x_i\le(\log 2)\ell^K,
        \qquad
        \frac{\log(\log\log x_i)}{\log 2}=O_K(\log\ell).
$$
The lower bound for $(\log u_i)/\log 2$ is of order $\ell^{K-2}$, so it eventually exceeds $\log(\log\log x_i)/\log 2$, uniformly in $i$.  Thus $u_i\ge\log\log x_i$ for all sufficiently large $\ell$, and hence $y_0\le x_i^{1/\log\log x_i}$.  By \eqref{eq:orthogonality}, Lemma \ref{lem:smooth-number-rankin}, and monotonicity in the smoothness bound,
$$
\begin{aligned}
        \E[|\Psi_f(x_i,y_0)|^2]
        &=\#\{n\le x_i:P(n)\le y_0\}\\
        &\ll
        x_i(\log x_i)^{-c\log\log\log x_i}.
\end{aligned}
$$
Chebyshev and a union over the $O(\exp(C_{c_0}L))$ indices $i$ satisfying $X_{\ell-1}<x_i\le X_\ell$ give
$$
        \P(\CB_\ell^{(0)})
        \ll
        \exp(C_{c_0}L-cL\log\ell)\ell^{O_\alpha(1)},
$$
which is summable.

For the repeated-prime term, Lemma \ref{lem:repeated-prime-probability} gives a summable contribution from $\bar\CX_\ell$.  Since $m=L$, $J\ll_KL$, and $R_\ell^2\asymp L^{2\alpha}$,
$$
        \frac{J^2}{\ell^{10K}R_\ell^2}
        \ll_K
        \frac{L^2}{\ell^{10K}L^{2\alpha}}
        =\ell^{-K(8+2\alpha)}.
$$
Equation \eqref{eq:repeated-general-R} is therefore bounded by
$$
        \exp\left(
        C_{c_0}L-K(8+2\alpha)L\log\ell+O_K(\log L)
        \right),
$$
which is summable for $\alpha>1/4$.
\end{proof}

\subsection{Smoothing the quadratic variation}\label{sec:smoothing}
It remains to treat the event $\CB_\ell^{(1)}$.  The quadratic variation of the martingale $M_f^{(1)}(x_i)$ is
$$
        V_\ell(x_i;f)
        =
        \sum_{y_0<p\le y_J}
        \left|\Psi_f'\left(\frac{x_i}{p},p\right)\right|^2,
$$
and we also need the parameters
$$
        D=8m_0^2-8m_0+4,
        \qquad
        X=(\log x_i)^D,
        \qquad
        T(\ell)=\ell^{10},
        \qquad
        T_1(\ell)=\frac{T(\ell)}{\ell\log\ell}.
$$
Thus $\log X\asymp_{m_0}L$ whenever $X_{\ell-1}<x_i\le X_\ell$.  The following estimate for reciprocal sums of primes in the very short multiplicative intervals $(v/(1+1/X),v]$ is used repeatedly in the auxiliary estimates.  Our blocks extend further than those in Caich's smoothing step \cite{Caich}, so the lemma covers the full range of $v$.

\begin{lem}\label{lem:short-reciprocal-primes}
For all sufficiently large $\ell$,
\begin{equation}\label{eq:short-reciprocal-primes}
        \sum_{v/(1+1/X)<p\le v}\frac1p
        \ll_{K,m_0}\frac1{X\log v},
\end{equation}
uniformly for $X_{\ell-1}<x_i\le X_\ell$ and $y_0\le v\le y_J(1+1/X)$.
\end{lem}

\begin{proof}
Set $u=v/(1+1/X)$.  Lemma \ref{lem:quantitative-mertens} gives
$$
\begin{aligned}
        \sum_{u<p\le v}\frac1p
        &=
        \log\log v-\log\log u
        +O\left(e^{-c\sqrt{\log u}}\right)\\
        &=
        -\log\left(
        1-\frac{\log(1+1/X)}{\log v}
        \right)
        +O\left(e^{-c\sqrt{\log u}}\right).
\end{aligned}
$$
The main term is $O(1/(X\log v))$.  Moreover, $\log(X\log v)\ll_{K,m_0}L$, while
$$
        \log u\ge\tfrac12\log y_0
        =2^{L-K\ell^{K-1}-1}
$$
for large $\ell$.  The error term is therefore $O_{K,m_0}(1/(X\log v))$, uniformly through the final block.
\end{proof}

For $z\le x$, let
$$
        \Psi_f'(x,z,y)=
        \sum_{\substack{z<n\le x\\P(n)<y}}f(n).
$$
Define
$$
        W_\ell(x_i;f)=
        \sum_{y_0<p\le y_J}\frac Xp
        \int_p^{p(1+1/X)}
        \left|
        \Psi_f'\left(\frac{x_i}{p},\frac{x_i}{t},p\right)
        \right|^2\dee t.
$$
For $1\le j\le J$, let
$$
        M_\ell(x_i,y_j;f)
        =
        \frac1{\log y_j}
        \int_{x_i/y_j}^{x_i/y_{j-1}}
        \left|\Psi_f\left(z,\frac{x_i}{z}\right)\right|^2
        \frac{\dee z}{z^2},
$$
$$
\begin{aligned}
        \lambda_\ell^{(2)}(x_i,y_j;f)
        &=
        \frac1{\log y_j}
        \int_{x_i/y_j}^{x_i/y_{j-1}}
        \sup_{\frac{x_i}{z(1+1/X)}\le q\le x_i/z}
        \left|
        \sum_{\substack{n\le z\\
        x_i/(z(1+1/X))<P(n)<q}}f(n)
        \right|^2
        \frac{\dee z}{z^2},
\end{aligned}
$$
and
$$
\begin{aligned}
        \lambda_\ell^{(3)}(x_i,y_j;f)
        &=
        \frac1{\log y_j}
        \int_{x_i/y_j}^{x_i/y_{j-1}}
        \left|
        \sum_{\substack{n\le z\\
        x_i/(z(1+1/X))<P(n)\le x_i/z}}f(n)
        \right|^2
        \frac{\dee z}{z^2}.
\end{aligned}
$$
For $j=0$ we set
$$
        \lambda_\ell^{(2)}(x_i,y_0;f)
        =
        \lambda_\ell^{(3)}(x_i,y_0;f)=0.
$$
The supremum in $\lambda_\ell^{(2)}(x_i,y_j;f)$ is measurable because the inner sum is constant between consecutive primes.

The edge terms are
$$
\begin{aligned}
        L_\ell^{(12)}(x_i;f)
        &=
        \sum_{\substack{1\le j\le J\\
        \log x_i/\log y_{j-1}>\ell^{100K}}}
        \int_{x_i/y_j}^{x_i/y_{j-1}}
        X\sum_{x_i/(z(1+1/X))<p\le x_i/z}
        \frac{|\Psi_f'(z,p)|^2}{p}\frac{\dee z}{z^2},
\end{aligned}
$$
and
$$
\begin{aligned}
        L_\ell^{(2)}(x_i;f)
        &=
        \sum_{\substack{1\le j\le J\\x_i\ge y_{j-1}}}
        \int_{x_i/(y_j(1+1/X))}^{x_i/y_j}
        X\sum_{\max\{x_i/(z(1+1/X)),y_{j-1}\}<p\le y_j}
        \frac{|\Psi_f'(z,p)|^2}{p}\frac{\dee z}{z^2}.
\end{aligned}
$$

\begin{lem}\label{lem:lambda-two}
One has
$$
\sum_{\ell\ge2}
\P\left(
\sup_{\substack{X_{\ell-1}<x_i\le X_\ell\\0\le j\le J}}
\lambda_\ell^{(2)}(x_i,y_j;f)
>
T_1(\ell)L^{-1/2}
\right)<\infty.
$$
\end{lem}

\begin{proof}
Fix $i,j,z$ with $1\le j\le J$ and $x_i/y_j\le z\le x_i/y_{j-1}$, and let
$$
        a=\frac{x_i}{z(1+1/X)},
        \qquad b=\frac{x_i}{z}.
$$
For a prime $p\in(a,b]$, define
$$
        A_p(z)=\sum_{\substack{n\le z\\a<P(n)\le p}}f(n).
$$
If $p^-$ is the preceding prime, then
\begin{equation}\label{eq:lambda-two-recurrence}
        A_p(z)-A_{p^-}(z)
        =
        \sum_{\substack{h\ge1\\p^h\le z}}
        f(p)^h
        \Psi_f'\left(\frac{z}{p^h},p\right).
\end{equation}
The coefficient of each $f(p)^h$ is measurable before $f(p)$ is revealed and $\E[f(p)^h]=0$.  Thus $(A_p(z))$ is a martingale, and conditional Doob at exponent $4$ gives
\begin{equation}\label{eq:lambda-two-doob}
\E\left[\sup_{a\le q\le b}
\left|
\sum_{\substack{n\le z\\a<P(n)<q}}f(n)
\right|^4\right]
\ll
\E\left[\left|
\sum_{\substack{n\le z\\a<P(n)\le b}}f(n)
\right|^4\right].
\end{equation}
Lemma \ref{lem:hypercontractive} bounds the right-hand side by
$$
        \left(
        \sum_{\substack{n\le z\\a<P(n)\le b}}\tau_3(n)
        \right)^2.
$$
Writing $p=P(n)$ and $n=pk$, using $\tau_3(pk)\le3\tau_3(k)$ and
$$
        \sum_{k\le w}\tau_3(k)
        \le w\left(\sum_{d\le w}\frac1d\right)^2
        \ll w(\log(2w))^2,
$$
we obtain
\begin{equation}\label{eq:lambda-two-fourth}
        \E\left[\left|
        \sum_{\substack{n\le z\\a<P(n)\le b}}f(n)
        \right|^4\right]
        \ll
        z^2(\log(2z))^4
        \left(\sum_{a<p\le b}\frac1p\right)^2.
\end{equation}
Since $b\in[y_{j-1},y_j]$, Lemma \ref{lem:short-reciprocal-primes} gives
$$
        \E\left[\left|
        \sum_{\substack{n\le z\\a<P(n)\le b}}f(n)
        \right|^4\right]
        \ll
        \frac{z^2(\log(2z))^4}
        {X^2(\log y_{j-1})^2}.
$$

Let $A_j=x_i/y_j$ and $B_j=x_i/y_{j-1}$, and let
$$
        S_2(z)=
        \sup_{\frac{x_i}{z(1+1/X)}\le q\le x_i/z}
        \left|
        \sum_{\substack{n\le z\\
        x_i/(z(1+1/X))<P(n)<q}}f(n)
        \right|^2.
$$
All sums defining $S_2(z)$ vanish for $0<z<1$.  Thus, if $B_j\le1$, then $\lambda_\ell^{(2)}(x_i,y_j;f)=0$.  Suppose that $B_j>1$, and set $C_j=\max\{A_j,1\}$.  Then
$$
        \lambda_\ell^{(2)}(x_i,y_j;f)
        =
        \frac1{\log y_j}
        \int_{C_j}^{B_j}S_2(z)\frac{\dee z}{z^2}.
$$
The required weighted Cauchy--Schwarz inequality is
$$
\begin{aligned}
        \left(\int_{C_j}^{B_j}S_2(z)\frac{\dee z}{z^2}\right)^2
        &=
        \left(\int_{C_j}^{B_j}
        \frac{S_2(z)}{z^{3/2}}\frac{\dee z}{z^{1/2}}\right)^2\\
        &\le
        \left(\int_{C_j}^{B_j}\frac{\dee z}{z}\right)
        \left(\int_{C_j}^{B_j}S_2(z)^2\frac{\dee z}{z^3}\right).
\end{aligned}
$$
On the effective range $C_j\le z\le B_j$, one has $1\le z\le x_i$ and hence $\log(2z)\ll\log x_i$.  Moreover,
$$
        \log(B_j/C_j)
        \le\log(B_j/A_j)
        =\log(y_j/y_{j-1}).
$$
Taking expectations and using \eqref{eq:lambda-two-doob} and the fourth-moment estimate above, we obtain
$$
\begin{aligned}
        \E[\lambda_\ell^{(2)}(x_i,y_j;f)^2]
        &\ll
        \frac{\log(B_j/C_j)}{(\log y_j)^2}
        \int_{C_j}^{B_j}
        \frac{z^2(\log(2z))^4}
        {X^2(\log y_{j-1})^2}\frac{\dee z}{z^3}\\
        &\ll
        \frac{(\log(y_j/y_{j-1}))^2(\log x_i)^4}
        {X^2(\log y_j)^2(\log y_{j-1})^2}.
\end{aligned}
$$
By \eqref{eq:block-ratio},
$$
        \log\left(\frac{y_j}{y_{j-1}}\right)
        =(e^{1/\ell}-1)\log y_{j-1}
        \ll\frac{\log y_{j-1}}{\ell}.
$$
Consequently,
$$
        \E[\lambda_\ell^{(2)}(x_i,y_j;f)^2]
        \ll
        \frac{(\log x_i)^4}
        {\ell^2X^2(\log y_{j-1})^2}
        \le
        \frac{(\log x_i)^4}
        {X^2(\log y_0)^2}.
$$
Markov's inequality, $J\ll_KL$, and the threshold $T_1(\ell)L^{-1/2}$ give
$$
\begin{aligned}
&\P\left(
\sup_{\substack{X_{\ell-1}<x_i\le X_\ell\\1\le j\le J}}
\lambda_\ell^{(2)}(x_i,y_j;f)>T_1(\ell)L^{-1/2}
\right)\\
&\qquad\ll_K
\frac{L^2}{T_1(\ell)^2(\log y_0)^2}
\sum_{X_{\ell-1}<x_i\le X_\ell}
\frac{(\log x_i)^4}{X^2}.
\end{aligned}
$$
Since $X=(\log x_i)^D$ and $c_0(2D-4)>1$, the sum over $i$ is uniformly bounded.  The remaining factor $(\log y_0)^{-2}$ makes the probabilities summable.  The endpoint $j=0$ vanishes by definition.
\end{proof}
The next proposition collects the deterministic smoothing inequality and the remaining auxiliary estimates.  Outside events of summable probability, each auxiliary quantity contributes $O(T(\ell)L^{-1/2})$.

\begin{prop}\label{prop:caich-smoothing}
For all sufficiently large $\ell$ and every $X_{\ell-1}<x_i\le X_\ell$,
\begin{equation}\label{eq:smoothing}
\begin{aligned}
        \frac{V_\ell(x_i;f)}{x_i}
        \ll{}&
        \ell\log\ell\sup_{1\le j\le J}M_\ell(x_i,y_j;f)
        +\ell\log\ell\sum_{k=2}^3
        \sup_{0\le j\le J}\lambda_\ell^{(k)}(x_i,y_j;f)\\
        &+L_\ell^{(12)}(x_i;f)+L_\ell^{(2)}(x_i;f)
        +\frac{W_\ell(x_i;f)}{x_i}.
\end{aligned}
\end{equation}
Moreover,
\begin{equation}\label{eq:auxiliary-events}
\begin{aligned}
&\sum_{\ell\ge2}
\P\left(
\sup_{\substack{X_{\ell-1}<x_i\le X_\ell\\0\le j\le J}}
\lambda_\ell^{(k)}(x_i,y_j;f)>
T_1(\ell)L^{-1/2}
\right)<\infty
\qquad(k=2,3),\\
&\sum_{\ell\ge2}
\P\left(
\sup_{X_{\ell-1}<x_i\le X_\ell}
L_\ell^{(12)}(x_i;f)>T(\ell)L^{-1/2}
\right)<\infty,\\
&\sum_{\ell\ge2}
\P\left(
\sup_{X_{\ell-1}<x_i\le X_\ell}
L_\ell^{(2)}(x_i;f)>T(\ell)L^{-1/2}
\right)<\infty,\\
&\sum_{\ell\ge2}
\P\left(
\sup_{X_{\ell-1}<x_i\le X_\ell}
\frac{W_\ell(x_i;f)}{x_i}>L^{-1/2}
\right)<\infty.
\end{aligned}
\end{equation}
\end{prop}

\begin{proof}
The pointwise inequality is the deterministic estimate obtained from the inequality immediately preceding equation~(24) and equation~(24) itself in \cite[Section~7.1]{Caich}.  With the strict and weak smoothness conventions specified above, its principal term, two $\lambda$-terms, two boundary terms, and $W$-term are exactly those in \eqref{eq:smoothing}.  The derivation is carried out block by block and pointwise in $f$, and it remains valid for the final block, since primes exceeding $x_i$ contribute zero, while every short reciprocal-prime estimate needed up to $y_J$ is given uniformly by Lemma \ref{lem:short-reciprocal-primes}.  We prove all probabilistic auxiliary bounds here.

The case $k=2$ in the first line of \eqref{eq:auxiliary-events} is Lemma \ref{lem:lambda-two}.  For $k=3$, let
$$
        B(z)=
        \sum_{\substack{n\le z\\
        x_i/(z(1+1/X))<P(n)\le x_i/z}}f(n)
$$
and $S_3(z)=|B(z)|^2$.  The calculation \eqref{eq:lambda-two-fourth}, without the preceding Doob step, gives
$$
        \E[S_3(z)^2]
        =
        \E[|B(z)|^4]
        \ll
        \frac{z^2(\log(2z))^4}
        {X^2(\log y_{j-1})^2}.
$$
With $A_j=x_i/y_j$ and $B_j=x_i/y_{j-1}$ as above, all sums defining $S_3(z)$ vanish for $0<z<1$.  If $B_j\le1$, then $\lambda_\ell^{(3)}(x_i,y_j;f)=0$.  Otherwise, with $C_j=\max\{A_j,1\}$,
$$
        \left(\int_{C_j}^{B_j}S_3(z)\frac{\dee z}{z^2}\right)^2
        \le
        \left(\int_{C_j}^{B_j}\frac{\dee z}{z}\right)
        \left(\int_{C_j}^{B_j}S_3(z)^2\frac{\dee z}{z^3}\right).
$$
On this effective range $\log(2z)\ll\log x_i$, and $\log(B_j/C_j)\le\log(y_j/y_{j-1})$.  Therefore
$$
\begin{aligned}
        \E[\lambda_\ell^{(3)}(x_i,y_j;f)^2]
        &\ll
        \frac{(\log(y_j/y_{j-1}))^2(\log x_i)^4}
        {X^2(\log y_j)^2(\log y_{j-1})^2}\\
        &\ll
        \frac{(\log x_i)^4}
        {\ell^2X^2(\log y_{j-1})^2}
        \le
        \frac{(\log x_i)^4}
        {X^2(\log y_0)^2}.
\end{aligned}
$$
The same Markov and union calculation as in Lemma \ref{lem:lambda-two} proves the $k=3$ assertion.

We next prove the high-moment estimate for $W_\ell(x_i;f)$.  Write $m=m_0$ and $k=2m-1$.  Minkowski's integral inequality and Lemma \ref{lem:hypercontractive} give
\begin{align}
        \|W_\ell(x_i;f)\|_m
        &\le
        \sum_{y_0<p\le y_J}\frac Xp
        \int_p^{p(1+1/X)}
        \left\|
        \Psi_f'\left(\frac{x_i}{p},\frac{x_i}{t},p\right)
        \right\|_{2m}^2\dee t \notag\\
        &\le
        \sum_{y_0<p\le y_J}\frac Xp
        \int_p^{p(1+1/X)}
        \sum_{\substack{x_i/t<n\le x_i/p\\P(n)<p}}
        \tau_k(n)\dee t.                       \label{eq:W-sequence}
\end{align}
Dropping $P(n)<p$ only enlarges the nonnegative sum.  For fixed $n,p$, the measure of the contributing $t$-set is at most $p/X$, so its weight after multiplication by $X/p$ is at most $1$.  A contribution is possible only when
$$
        \frac{x_i/n}{1+1/X}<p\le\frac{x_i}{n}.
$$
If $v=x_i/n$, Lemma \ref{lem:short-reciprocal-primes} and $\#\{p\in I\}\le v\sum_{p\in I}1/p$ give
$$
        \#\left\{p:\frac{v}{1+1/X}<p\le v\right\}
        \ll\frac{v}{X\log v}.
$$
Whenever this set meets $(y_0,y_J]$, its parameter satisfies $y_0<v\le y_J(1+1/X)$, so the bound is uniform.  Interchanging the nonnegative sums in \eqref{eq:W-sequence} therefore gives
$$
\begin{aligned}
        \|W_\ell(x_i;f)\|_m
        &\ll
        \frac{x_i}{X}
        \sum_{n\le x_i/y_0}
        \frac{\tau_k(n)}{n\log(x_i/n)}\\
        &\le
        \frac{x_i}{X\log y_0}
        \sum_{n\le x_i}\frac{\tau_k(n)}n.
\end{aligned}
$$
Since
$$
        \sum_{n\le x}\frac{\tau_k(n)}n
        \le
        \left(\sum_{d\le x}\frac1d\right)^k
        \ll_k(\log x)^k,
$$
and $D>k+1$, we obtain, uniformly in $i$ and $\ell$,
\begin{equation}\label{eq:W-moment}
        \E[W_\ell(x_i;f)^{m_0}]
        \ll_{m_0}\left(\frac{x_i}{\log x_i}\right)^{m_0}.
\end{equation}
Markov gives
$$
\P\left(
\sup_{X_{\ell-1}<x_i\le X_\ell}
\frac{W_\ell(x_i;f)}{x_i}>L^{-1/2}
\right)
\ll_{m_0}
\sum_{X_{\ell-1}<x_i\le X_\ell}
\frac{L^{m_0/2}}{(\log x_i)^{m_0}}.
$$
The sum over all $i$ converges because $\log x_i\asymp i^{c_0}$, $c_0m_0>1$, and the factor $L^{m_0/2}$ grows only polylogarithmically in $i$.  This proves the $W_\ell$ assertion.

For $L_\ell^{(12)}$, consider a block occurring in its definition.  For $z\in[x_i/y_j,x_i/y_{j-1}]$ and a prime
$$
        \frac{x_i}{z(1+1/X)}<p\le\frac{x_i}{z},
$$
we have, by \eqref{eq:block-ratio},
$$
\begin{aligned}
        \frac{\log z}{\log p}
        &\ge
        \frac{\log x_i-\log y_j}{\log y_j}\\
        &=
        \frac{\log x_i}{\log y_j}-1
        >
        e^{-1/\ell}\ell^{100K}-1
        \ge\frac12\ell^{100K}
\end{aligned}
$$
for sufficiently large $\ell$.  Orthogonality gives
$$
        \E[|\Psi_f'(z,p)|^2]
        =\#\{n\le z:P(n)<p\}
        \le\Psi(z,p).
$$
Lemma \ref{lem:smooth-number-rankin} applies uniformly because $\log\log p\ll_KL$ and $\log p$ is doubly exponential in $\ell$, and it yields the bound
\begin{equation}\label{eq:L12-smooth-decay}
        \E[|\Psi_f'(z,p)|^2]
        \le z\exp(-\ell^{99K})
\end{equation}
for all sufficiently large $\ell$.  Hence Lemma \ref{lem:short-reciprocal-primes}, followed by $v=x_i/z$, gives
$$
\begin{aligned}
        \E[L_\ell^{(12)}(x_i;f)]
        &\ll
        e^{-\ell^{99K}}
        \sum_{j=1}^J
        \int_{x_i/y_j}^{x_i/y_{j-1}}
        \frac{\dee z}{z\log(x_i/z)}\\
        &=
        e^{-\ell^{99K}}
        \sum_{j=1}^J
        \int_{y_{j-1}}^{y_j}\frac{\dee v}{v\log v}\\
        &=
        \frac{J}{\ell}e^{-\ell^{99K}}
        \ll_K L e^{-\ell^{99K}}.
\end{aligned}
$$
In particular,
\begin{equation}\label{eq:L12-expectation}
        \E[L_\ell^{(12)}(x_i;f)]
        \ll_K L\exp(-\ell^{99K}/2).
\end{equation}
The weaker form \eqref{eq:L12-expectation} follows immediately from the preceding estimate.  Since there are at most $\exp(C_{c_0}L)$ indices $i$ satisfying $X_{\ell-1}<x_i\le X_\ell$,
$$
\begin{aligned}
&\P\left(
\sup_{X_{\ell-1}<x_i\le X_\ell}
L_\ell^{(12)}(x_i;f)>T(\ell)L^{-1/2}
\right)\\
&\qquad\ll
\frac{L^{3/2}}{T(\ell)}
\exp(C_{c_0}L-\ell^{99K}/2),
\end{aligned}
$$
which is summable.

Finally, $\E[|\Psi_f'(z,p)|^2]\le z$.  In the definition of $L_\ell^{(2)}$, one has
$$
        y_j\le\frac{x_i}{z}\le y_j(1+1/X),
$$
and every prime in the inner sum lies in $(y_j/(1+1/X),y_j]$.  Lemma \ref{lem:short-reciprocal-primes} gives
$$
\begin{aligned}
        \E[L_\ell^{(2)}(x_i;f)]
        &\ll
        \sum_{j=1}^J
        \int_{x_i/(y_j(1+1/X))}^{x_i/y_j}
        \frac{\dee z}{z\log y_j}\\
        &\ll
        \frac1X\sum_{j=1}^J\frac1{\log y_j}
        \ll\frac{J}{X\log y_0}.
\end{aligned}
$$
Thus
\begin{equation}\label{eq:L2-expectation}
        \E[L_\ell^{(2)}(x_i;f)]
        \ll\frac{J}{X\log y_0}.
\end{equation}
Markov and a union bound reduce the sum of the exceptional probabilities to
$$
        \sum_{\ell\ge2}
        \frac{L^{1/2}J}{T(\ell)\log y_0}
        \sum_{X_{\ell-1}<x_i\le X_\ell}(\log x_i)^{-D},
$$
which converges because $c_0D>1$.  This completes \eqref{eq:auxiliary-events}.
\end{proof}

\section{The Steinhaus low moment}\label{sec:lowmoment}

For $Y\ge3$, define
$$
        F_Y(s)=
        \prod_{p\le Y}\left(1-\frac{f(p)}{p^s}\right)^{-1}.
$$
For $\mathrm{Re}(s)>0$, this is the Dirichlet series of $f$ restricted to $Y$-smooth integers.  We require the following low moment of its full-line mean square.
\begin{lem}\label{lem:steinhaus-harper}
Uniformly for $Y\ge3$,
$$
        \E\left[
        \left(
        \frac1{\log Y}
        \int_{-\infty}^{\infty}
        \left|
        \frac{F_Y(\tfrac12+it)}{\tfrac12+it}
        \right|^2\dee t
        \right)^{2/3}
        \right]
        \ll(\log\log Y)^{-1/3}.
$$
\end{lem}

\begin{proof}
In Section 4.1 of the proof of the upper bound in \cite[Theorem 1]{HarperLow}, Harper proves
$$
\E\left[
\left(
\frac{e^k(1-q)\sqrt{\log\log x}}{\log x}
\int_{-1/2}^{1/2}
\left|
F_k\left(\tfrac12-\frac{k}{\log x}+it\right)
\right|^2\dee t
\right)^q
\right]\ll1,
$$
uniformly for
$$
        0\le k\le\lfloor\log\log\log x\rfloor,
        \qquad
        \frac23\le q\le
        1-\frac1{\sqrt{\log\log x}},
$$
where
$$
        F_k(s)=
        \prod_{p\le x^{e^{-(k+1)}}}
        \left(1-\frac{f(p)}{p^s}\right)^{-1}.
$$
Take $x=Y^e$, $k=0$ and $q=2/3$.  Then $F_0(s)=F_Y(s)$, and
\begin{equation}\label{eq:harper-interval}
        \E\left[
        \left(
        \frac1{\log Y}
        \int_{-1/2}^{1/2}|F_Y(\tfrac12+it)|^2\dee t
        \right)^{2/3}
        \right]
        \ll(\log\log Y)^{-1/3}.
\end{equation}

For $N\in\Z$, the variables $f(p)p^{-iN}$ are again independent and uniform on $\T$.  Therefore the process on $[N-1/2,N+1/2]$ has the same law as that on $[-1/2,1/2]$, so \eqref{eq:harper-interval} holds uniformly on every translated unit interval.  Since
$$
        |\tfrac12+it|^{-2}
        \ll(1+|N|)^{-2}
        \qquad
        (N-\tfrac12\le t\le N+\tfrac12)
$$
and $u^{2/3}$ is subadditive,
$$
\begin{aligned}
&\E\left[
\left(
\frac1{\log Y}
\int_{-\infty}^{\infty}
\left|
\frac{F_Y(\tfrac12+it)}{\tfrac12+it}
\right|^2\dee t
\right)^{2/3}
\right]\ll
(\log\log Y)^{-1/3}
\sum_{N\in\Z}(1+|N|)^{-4/3},
\end{aligned}
$$
and the series converges.  Bounded $Y$ are absorbed into the constant using
$$
        \E[|F_Y(\tfrac12+it)|^2]
        =\prod_{p\le Y}(1-1/p)^{-1}\ll\log Y.
$$
\end{proof}

\section{Fixed conditional block moments}\label{sec:block}
Recall that $\CF_y$ is the $\sigma$-algebra generated by the values $f(p)$, $p\le y$.  For $0\le j\le J$, abbreviate $F_j(s)=F_{y_j}(s)$, and define the tilted, normalised Euler product integrals
\begin{equation}\label{eq:Ij}
        I_j=\left(\frac{\log y_j}{\log y_0}\right)^{-1/L}
        \frac{1}{\log y_j}\int_{-\infty}^{\infty}
        \left|\frac{F_j(\tfrac12+it)}{\tfrac12+it}\right|^{2}\dee t,
\end{equation}
which are nonnegative, $\CF_{y_j}$-measurable, and almost surely finite.  Indeed, for fixed $j$ the finite Euler product is bounded in $t$, so the integrand is $O_{f,j}((1+t^2)^{-1})$.  The tilt $(\log y_j/\log y_0)^{-1/L}$ decreases slowly in $j$.  It changes $I_j$ by only a bounded factor over all $J$ blocks, since
$$
        \left(\frac{\log y_J}{\log y_0}\right)^{1/L}=e^{J/(\ell L)}\ll_K1,
$$
and it makes the resulting sequence a supermartingale.

For $1\le j\le J$, let
\begin{equation*}
        \CQ_j=\{y_{j-1}\}\cup
        \{q:\ y_{j-1}<q\le y_j,\ q\ \mathrm{prime}\}
\end{equation*}
denote the set of smoothness cutoffs interior to the $j$-th block, and define the block variables
\begin{equation*}
        U_j=\frac{1}{\log y_j}\int_0^\infty
        \max_{q\in\CQ_j}|\Psi_f(z,q)|^{2}\frac{\dee z}{z^{2}},
        \qquad
        U_0=\frac{1}{\log y_0}\int_0^\infty
        |\Psi_f(z,y_0)|^{2}\frac{\dee z}{z^{2}}.
\end{equation*}
As explained in Section \ref{sec:outline}, $U_j$ dominates the principal terms $M_\ell(x_i,y_j;f)$ simultaneously for all test points $x_i$.  The maximum over $\CQ_j$ gives this uniformity, and Doob's inequality will remove it in the proof of Proposition \ref{prop:block-moment}.

The first lemma establishes the supermartingale property.  The proof accounts for three factors per block.  Integrating out the new block variables multiplies the integrand, in expectation, by the Euler factor
$$
        \prod_{y_{j-1}<p\le y_j}\left(1-\frac1p\right)^{-1}=\exp\left(\frac1\ell+o\left(\frac1{\ell L}\right)\right).
$$
The normalisation $1/\log y_j$ contributes $e^{-1/\ell}$ by \eqref{eq:block-ratio}, and the tilt contributes $e^{-1/(\ell L)}$.  The product of the three is $\exp(-1/(\ell L)+o(1/(\ell L)))<1$.

\begin{lem}\label{lem:supermartingale}
For all sufficiently large $\ell$ in terms of $K$, $(I_j)_{0\le j\le J}$ is a nonnegative supermartingale with respect to $(\CF_{y_j})_{0\le j\le J}$.
\end{lem}

\begin{proof}
Tonelli and Mertens give
\begin{equation}\label{eq:I0-integrability}
        \E[I_0]=
        \frac{2\pi}{\log y_0}
        \prod_{p\le y_0}(1-1/p)^{-1}\ll1.
\end{equation}
Conditioning on $\CF_{y_{j-1}}$ and expanding the Euler factors belonging to $(y_{j-1},y_j]$,
$$
        \E\left[\left|
        1-\frac{f(p)}{p^{1/2+it}}
        \right|^{-2}\right]
        =
        \sum_{a\ge0}p^{-a}
        =(1-1/p)^{-1}.
$$
Conditional Tonelli is legitimate because the integrand is nonnegative.  Thus
\begin{equation}\label{eq:Ij-conditional-multiplier}
        \E[I_j\mid\CF_{y_{j-1}}]=b_jI_{j-1},
\end{equation}
where, for every block including the last,
$$
        b_j=
        e^{-1/\ell}e^{-1/(\ell L)}
        \prod_{y_{j-1}<p\le y_j}(1-1/p)^{-1}.
$$
By Lemma \ref{lem:quantitative-mertens},
$$
        \sum_{y_{j-1}<p\le y_j}\frac1p
        =
        \frac1\ell+
        O(e^{-c\sqrt{\log y_0}})
$$
uniformly in $j$, while $\sum_{p>y_{j-1}}p^{-2}\ll y_{j-1}^{-1}$.  Hence
$$
        \log b_j
        =
        -\frac1{\ell L}
        +O(e^{-c\sqrt{\log y_0}})
        +O(y_{j-1}^{-1})
        \le-\frac1{2\ell L}
$$
for all sufficiently large $\ell$.  Therefore $b_j\le1$.  Equation \eqref{eq:I0-integrability} and induction show that every $I_j$ is integrable, so \eqref{eq:Ij-conditional-multiplier} proves the supermartingale property.
\end{proof}

\begin{prop}\label{prop:block-moment}
Let $r\ge2$ be fixed.  Uniformly for $1\le j\le J$, and for all sufficiently large $\ell$ in terms of $K$,
$$
        \E[U_j^r\mid\CF_{y_{j-1}}]
        \ll_r I_{j-1}^r
        \qquad\text{almost surely}.
$$
\end{prop}

\begin{proof}
Fix $j$ and condition on $\CF_{y_{j-1}}$.  Reveal the primes in $(y_{j-1},y_j]$ in increasing order.  If $p$ follows a cutoff $q$, then
$$
        \Psi_f(z,p)
        =
        \Psi_f(z,q)
        +\sum_{\substack{k\ge1\\p^k\le z}}
        f(p)^k\Psi_f(z/p^k,q).
$$
The powers of the newly revealed prime variable have mean zero, so $q\mapsto\Psi_f(z,q)$ is a martingale along $\CQ_j$.  Conditional Doob at exponent $2r$ gives
\begin{equation}\label{eq:doob-step}
        \E\left[
        \max_{q\in\CQ_j}|\Psi_f(z,q)|^{2r}
        \mid\CF_{y_{j-1}}
        \right]
        \ll_r
        \E\left[
        |\Psi_f(z,y_j)|^{2r}
        \mid\CF_{y_{j-1}}
        \right].
\end{equation}

For $Z\ge1$, let
$$
        U_{j,Z}=
        \frac1{\log y_j}
        \int_1^Z
        \max_{q\in\CQ_j}|\Psi_f(z,q)|^2\frac{\dee z}{z^2}.
$$
Conditional Minkowski and \eqref{eq:doob-step} give
\begin{equation}\label{eq:minkowski}
\begin{aligned}
        \E[U_{j,Z}^r\mid\CF_{y_{j-1}}]^{1/r}
        \ll_r
        \frac1{\log y_j}
        \int_1^Z
        \E[
        |\Psi_f(z,y_j)|^{2r}
        \mid\CF_{y_{j-1}}
        ]^{1/r}
        \frac{\dee z}{z^2}.
\end{aligned}
\end{equation}
Every $y_j$-smooth integer factors uniquely as $n=dm$, where every prime factor of $d$ lies in $(y_{j-1},y_j]$ and $P(m)\le y_{j-1}$.  Hence
$$
        \Psi_f(z,y_j)
        =
        \sum_{\substack{d\le z\\
        p\mid d\Rightarrow y_{j-1}<p\le y_j}}
        f(d)\Psi_f(z/d,y_{j-1}).
$$
Applying Lemma \ref{lem:hypercontractive} conditionally to the prime coordinates in $(y_{j-1},y_j]$ gives
$$
\begin{aligned}
&\E[
|\Psi_f(z,y_j)|^{2r}
\mid\CF_{y_{j-1}}
]^{1/r}\le
\sum_{\substack{d\le z\\
p\mid d\Rightarrow y_{j-1}<p\le y_j}}
\tau_{2r-1}(d)
|\Psi_f(z/d,y_{j-1})|^2.
\end{aligned}
$$
Substitution into \eqref{eq:minkowski}, Tonelli, and $z=du$ yield
$$
\begin{aligned}
        \E[U_{j,Z}^r\mid\CF_{y_{j-1}}]^{1/r}
        &\ll_r
        \frac1{\log y_j}
        \left(
        \sum_{\substack{d\ge1\\
        p\mid d\Rightarrow y_{j-1}<p\le y_j}}
        \frac{\tau_{2r-1}(d)}d
        \right)
        \int_0^\infty
        |\Psi_f(u,y_{j-1})|^2\frac{\dee u}{u^2}.
\end{aligned}
$$
The divisor sum equals
$$
        \prod_{y_{j-1}<p\le y_j}
        (1-1/p)^{-(2r-1)}
        \ll_r1,
$$
uniformly in $j$, because the reciprocal-prime mass of the block is $1/\ell+o(1/\ell)$.

Parseval gives
$$
        \int_0^\infty
        |\Psi_f(u,y_{j-1})|^2\frac{\dee u}{u^2}
        =
        \frac1{2\pi}
        \int_{-\infty}^{\infty}
        \left|
        \frac{F_{j-1}(\tfrac12+it)}
        {\tfrac12+it}
        \right|^2\dee t.
$$
Consequently,
\begin{align*}
&\frac1{\log y_j}
\int_0^\infty
|\Psi_f(u,y_{j-1})|^2\frac{\dee u}{u^2}=
\frac1{2\pi}
\frac{\log y_{j-1}}{\log y_j}
\left(\frac{\log y_{j-1}}{\log y_0}\right)^{1/L}
I_{j-1}.
\end{align*}
The residual factor is
$$
        e^{-1/\ell}
        \exp\left(\frac{j-1}{\ell L}\right).
$$
By \eqref{eq:J-count}, after increasing the lower threshold for $\ell$ in terms of $K$, this factor is at most $2$, uniformly for every $1\le j\le J$, including the final block.  Thus the constant introduced here is absolute, and in particular does not depend on $K$.  Together with the divisor-product bound, this shows that
$$
        \E[U_{j,Z}^r\mid\CF_{y_{j-1}}]^{1/r}
        \ll_r I_{j-1}
$$
uniformly in $Z$.  Since $U_{j,Z}\uparrow U_j$, conditional monotone convergence completes the proof.
\end{proof}

We now combine Lemmas \ref{lem:steinhaus-harper} and \ref{lem:supermartingale} with Proposition \ref{prop:block-moment} to bound the principal term of the smoothing inequality.  This proposition replaces the conditional first-moment step in \cite[Proposition 7.6]{Caich} by a fixed conditional $r$th-moment estimate on one prime block.
\begin{prop}\label{prop:principal}
There is a constant $C_r\ge1$ such that
$$
\sum_{\ell\ge2}
\P\left(
\sup_{X_{\ell-1}<x_i\le X_\ell}
\sup_{0\le j\le J}
M_\ell(x_i,y_j;f)
>
C_rT_1(\ell)L^{-1/2+1/r}
\right)<\infty,
$$
where
$$
        M_\ell(x_i,y_0;f)
        =
        \frac1{\log y_0}
        \int_0^{x_i/y_0}
        |\Psi_f(z,y_0)|^2\frac{\dee z}{z^2}.
$$
\end{prop}

\begin{proof}
For $1\le j\le J$ and $x_i/y_j\le z\le x_i/y_{j-1}$, the cutoff $x_i/z$ lies in $[y_{j-1},y_j]$.  Since a smooth sum is constant between consecutive primes, $M_\ell(x_i,y_j;f)\le U_j$.  Also $M_\ell(x_i,y_0;f)\le U_0$, and Parseval gives $U_0=I_0/(2\pi)$.

We stop the process $(I_j)$ at the level $T_1(\ell)^{1/2}L^{-1/2}$, and write
$$
        \CS_\ell=
        \left\{
        \max_{0\le j\le J}I_j
        \le T_1(\ell)^{1/2}L^{-1/2}
        \right\}.
$$
Since $u^{2/3}$ is increasing and concave, $(I_j^{2/3})$ is a nonnegative supermartingale.  Lemma \ref{lem:steinhaus-harper}, applied with $Y=y_0$, gives
\begin{equation}\label{eq:I0-two-thirds}
        \E[I_0^{2/3}]\ll_KL^{-1/3}.
\end{equation}
Doob's supermartingale inequality therefore yields
$$
        \P(\bar\CS_\ell)
        \ll_KT_1(\ell)^{-1/3}.
$$
Since $T_1(\ell)=\ell^9/\log\ell$, these probabilities are summable.

Let $H_{\ell,r}=T_1(\ell)L^{-1/2+1/r}$.  Choose $A_r$ so that Proposition \ref{prop:block-moment} gives
$$
        \E[U_j^r\mid\CF_{y_{j-1}}]\le A_rI_{j-1}^r,
$$
and take $C_r^r\ge A_r$.  Let
$$
        \CS_{\ell,j-1}
        =
        \left\{I_{j-1}\le T_1(\ell)^{1/2}L^{-1/2}\right\}.
$$
Since $\CS_{\ell,j-1}\in\CF_{y_{j-1}}$, conditional Markov gives
$$
\begin{aligned}
&\ind_{\CS_{\ell,j-1}}
\P(U_j>C_rH_{\ell,r}\mid\CF_{y_{j-1}})\\
&\qquad\le
\ind_{\CS_{\ell,j-1}}
\frac{A_rI_{j-1}^r}{C_r^rH_{\ell,r}^r}
\le T_1(\ell)^{-r/2}L^{-1}.
\end{aligned}
$$
Because $J\ll_KL$,
$$
        \P\left(
        \{\max_{1\le j\le J}U_j>C_rH_{\ell,r}\}
        \cap\CS_\ell
        \right)
        \ll_{r,K}T_1(\ell)^{-r/2},
$$
which is summable.  On $\CS_\ell$,
$$
        U_0\le T_1(\ell)^{1/2}L^{-1/2}
        \le H_{\ell,r}
$$
for all sufficiently large $\ell$.  Combining these bounds proves the proposition.
\end{proof}

Substituting Proposition \ref{prop:principal} into Proposition \ref{prop:caich-smoothing} bounds the quadratic variation.
\begin{prop}\label{prop:variance}
There is a constant $C_*=C_*(r,K,c_0,m_0)$ such that
$$
\sum_{\ell\ge1}
\P\left(
\sup_{X_{\ell-1}<x_i\le X_\ell}
\frac{V_\ell(x_i;f)}{x_i}
>
C_*T(\ell)L^{-1/2+1/r}
\right)<\infty.
$$
\end{prop}

\begin{proof}
Outside the summable exceptional event of Proposition \ref{prop:principal},
$$
        \ell\log\ell\sup_{1\le j\le J}
        M_\ell(x_i,y_j;f)
        \ll_r
        T(\ell)L^{-1/2+1/r}.
$$
Outside the events in \eqref{eq:auxiliary-events}, the two $\lambda$-terms, after multiplication by $\ell\log\ell$, and the terms $L_\ell^{(12)}$, $L_\ell^{(2)}$, and $W_\ell/x_i$ are all $O(T(\ell)L^{-1/2})$.  Since $L^{-1/2}\le L^{-1/2+1/r}$, \eqref{eq:smoothing} proves the claim.
\end{proof}

\section[The Steinhaus case of Proposition \ref{prop:fixed-r}]{The Steinhaus case of Proposition \ref{prop:fixed-r}}\label{sec:fixedr}

\begin{proof}[Proof of the Steinhaus case of Proposition \ref{prop:fixed-r}]
Let $\alpha=\frac14+\frac1{2r}+\eta$ and $R(x)=(\log\log x)^\alpha$.  Choose $K$ as in \eqref{eq:K-condition}, choose $c_0$ from Lemma \ref{lem:sparse-points} with $A=1$, and then choose $m_0>1/c_0$.  All finitely many excluded values of $\ell$ are absorbed into the random implied constant.

By \eqref{eq:largest-prime}, Proposition \ref{prop:test-point-reduction} and Lemma \ref{lem:B0-B2}, it remains to control $\CB_\ell^{(1)}$.  For $X_{\ell-1}<x_i\le X_\ell$, let
$$
        \Sigma_{\ell,i}=
        \left\{
        V_\ell(x_i;f)
        \le C_*x_iT(\ell)L^{-1/2+1/r}
        \right\}.
$$
Proposition \ref{prop:variance} gives
\begin{equation}\label{eq:variance-union}
        \sum_{\ell\ge1}
        \P\left(
        \bigcup_{X_{\ell-1}<x_i\le X_\ell}
        \bar\Sigma_{\ell,i}
        \right)<\infty.
\end{equation}

Enumerate the primes in $(y_0,y_J]$ as $p_1<\cdots<p_N$, set
$$
        \CH_0=\CF_{y_0},
        \qquad
        \CH_k=\CF_{y_0}\vee
        \sigma(f(p_h):1\le h\le k),
$$
and define
$$
        Z_k=f(p_k)\Psi_f'\left(\frac{x_i}{p_k},p_k\right),
        \qquad
        S_k=
        \left|\Psi_f'\left(\frac{x_i}{p_k},p_k\right)\right|.
$$
The strict condition $P(n)<p_k$ makes $S_k$ $\CH_{k-1}$-measurable even when $p_k<\sqrt{x_i}$.  Moreover $S_k\le x_i/p_k$, so every $Z_k$ is bounded, and $\E[Z_k\mid\CH_{k-1}]=0$.  By construction,
$$
        \sum_{k=1}^NS_k^2=V_\ell(x_i;f),
        \qquad
        \sum_{k=1}^NZ_k=M_f^{(1)}(x_i).
$$
Lemma \ref{lem:stopped-hoeffding}, with
$$
        u=\sqrt{x_i}R(x_i),
        \qquad
        T_i=C_*x_iT(\ell)L^{-1/2+1/r},
$$
gives
$$
\begin{aligned}
&\P\left(
|M_f^{(1)}(x_i)|\ge\sqrt{x_i}R(x_i),\
\Sigma_{\ell,i}
\right)\\
&\qquad\le
2\exp\left(
-\frac{R(x_i)^2}
{10C_*T(\ell)L^{-1/2+1/r}}
\right).
\end{aligned}
$$
By \eqref{eq:L-comparability},
$$
        \frac{R(x_i)^2}
        {T(\ell)L^{-1/2+1/r}}
        \asymp
        \frac{L^{1+2\eta}}{T(\ell)}
        =
        \ell^{K(1+2\eta)-10}.
$$
There are at most $\exp(C_{c_0}\ell^K)$ indices $i$ satisfying $X_{\ell-1}<x_i\le X_\ell$.  Hence
$$
\begin{aligned}
        \P(\CB_\ell^{(1)})
        &\le
        \P\left(
        \bigcup_{X_{\ell-1}<x_i\le X_\ell}
        \bar\Sigma_{\ell,i}
        \right)+
        2\exp\left(
        C_{c_0}\ell^K-c\ell^{K(1+2\eta)-10}
        \right).
\end{aligned}
$$
The second term is summable because $2K\eta>10$, and the first is summable by \eqref{eq:variance-union}.

The three terms in \eqref{eq:largest-prime} now satisfy
$$
        |M_f(x_i)|\le3\sqrt{x_i}R(x_i)
$$
almost surely for all sufficiently large $i$.  Proposition \ref{prop:test-point-reduction} extends this estimate from the test points to all large $x$, proving the Steinhaus case.
\end{proof}

\section{The Rademacher case}\label{sec:rademacher}

Let $f$ be given by \eqref{eq:rademacher-definition}, and keep the definitions of $M_f(x)$, $\Psi_f(x,y)$, $\Psi_f'(x,y)$, $x_i$, $X_\ell$, $y_j$, $J$ and $L$. We start by proving the Rademacher form of Lemma \ref{lem:hypercontractive}.
\begin{lem}\label{lem:rademacher-hypercontractive}
Let $m$ be a positive integer, and let $(a_n)$ be a sequence of complex numbers supported on squarefree integers such that $\sum_n|a_n|^2\tau_{2m-1}(n)<\infty$.  The random series $\sum_na_nf(n)$ is understood as the $L^2$ limit of its finite partial sums.  Then
$$
        \E\left[\left|\sum_na_nf(n)\right|^{2m}\right]
        \le
        \left(\sum_n|a_n|^2\tau_{2m-1}(n)\right)^m.
$$
More generally, let $\CP$ be a finite set of primes, let $\CG$ be independent of $(f(p))_{p\in\CP}$, and let $(a_d)$ be $\CG$-measurable and supported on squarefree integers composed only of primes in $\CP$.  Then
\begin{equation}\label{eq:rademacher-conditional-hypercontractive}
        \E\left[
        \left|\sum_da_df(d)\right|^{2m}
        \biggm|\CG
        \right]
        \le
        \left(\sum_d|a_d|^2\tau_{2m-1}(d)\right)^m.
\end{equation}
\end{lem}

\begin{proof}
Bonami's inequality for finite multilinear polynomials in independent signs gives
$$
        \E\left[\left|\sum_na_nf(n)\right|^{2m}\right]
        \le
        \left(
        \sum_n(2m-1)^{\omega(n)}|a_n|^2
        \right)^m.
$$
For squarefree $n$, $\tau_{2m-1}(n)=(2m-1)^{\omega(n)}$.  For the general sequence in the first assertion, orthogonality shows that the finite partial sums converge in $L^2$; applying the finite inequality to these partial sums and then using an almost surely convergent subsequence and Fatou's lemma gives the stated bound.  Freezing $\CG$ gives the conditional statement pointwise on the remaining sign space.
\end{proof}

The proof of Lemma \ref{lem:sparse-points} is unchanged, except that its $2m$th-moment estimate becomes
$$
        \E\left[\left|\sum_{n\in I}f(n)\right|^{2m}\right]
        \le
        \left(
        \sum_{n\in I}\mu^2(n)\tau_{2m-1}(n)
        \right)^m
        \le
        \left(
        \sum_{n\in I}\tau_{2m-1}(n)
        \right)^m.
$$
The largest-prime decomposition is simpler because there is no repeated-prime term.
\begin{lem}\label{lem:rademacher-decomposition}
For every $X_{\ell-1}<x_i\le X_\ell$,
\begin{equation}\label{eq:rademacher-largest-prime}
        M_f(x_i)=
        \Psi_f(x_i,y_0)
        +\sum_{y_0<p\le y_J}
        f(p)\Psi_f'\left(\frac{x_i}{p},p\right).
\end{equation}
\end{lem}

\begin{proof}
Every $n$ with $f(n)\ne0$ is squarefree.  If $P(n)>y_0$, write $n=pm$ with $p=P(n)$.  Then $P(m)<p$ and $f(n)=f(p)f(m)$.  This representation is unique, and no repeated-prime term remains.
\end{proof}

Write the prime sum in \eqref{eq:rademacher-largest-prime} as $M_f^{(1)}(x_i)$ and define
\begin{equation*}
        V_\ell(x_i;f)=\sum_{y_0<p\le y_J}
        \left|\Psi_f'\left(\frac{x_i}{p},p\right)\right|^2.
\end{equation*}
The smooth term is controlled by the first part of Lemma \ref{lem:B0-B2}.  Indeed, by \eqref{eq:rademacher-orthogonality},
\begin{equation*}
        \E[|\Psi_f(x_i,y_0)|^2]
        =\#\{n\le x_i:\mu^2(n)=1,\ P(n)\le y_0\}
        \le\#\{n\le x_i:P(n)\le y_0\},
\end{equation*}
so the same smooth-number estimate, Chebyshev inequality, union bound, and Borel--Cantelli argument apply.  The estimates of Section \ref{sec:repeated} devoted to $M_f^{(2)}$ are not needed.

\subsection{Transfer of the smoothing estimates}

Use the definitions of $M_\ell$, $\lambda_\ell^{(2)}$, $\lambda_\ell^{(3)}$, $L_\ell^{(12)}$, $L_\ell^{(2)}$, and $W_\ell$ from Section \ref{sec:smoothing}, with the Rademacher smooth sums.  The following is the Rademacher analogue of Proposition \ref{prop:caich-smoothing}.

\begin{prop}\label{prop:rademacher-smoothing}
For all sufficiently large $\ell$ and all $X_{\ell-1}<x_i\le X_\ell$,
\begin{equation}\label{eq:rademacher-smoothing}
\begin{aligned}
        \frac{V_\ell(x_i;f)}{x_i}
        \ll&\ \ell\log\ell\sup_{1\le j\le J}M_\ell(x_i,y_j;f)
        +\ell\log\ell\sum_{k=2}^{3}\sup_{0\le j\le J}
        \lambda_\ell^{(k)}(x_i,y_j;f)\\
        &+L_\ell^{(12)}(x_i;f)+L_\ell^{(2)}(x_i;f)
        +\frac{W_\ell(x_i;f)}{x_i}.
\end{aligned}
\end{equation}
Moreover, the summability estimates in \eqref{eq:auxiliary-events} hold for the Rademacher model.
\end{prop}

\begin{proof}
The deterministic decomposition leading to \eqref{eq:smoothing} uses the largest-prime representation of the prime-martingale term and deterministic inequalities for the resulting smooth sums.  Applying it to \eqref{eq:rademacher-largest-prime} gives \eqref{eq:rademacher-smoothing} with the same auxiliary quantities.

We prove the auxiliary estimates in the order used in Proposition \ref{prop:caich-smoothing}.  In the proof of Lemma \ref{lem:lambda-two}, if $p$ follows the preceding prime $p^-$, squarefree support replaces \eqref{eq:lambda-two-recurrence} by
\begin{equation*}
        A_p(z)-A_{p^-}(z)
        =f(p)\Psi_f'\left(\frac{z}{p},p\right).
\end{equation*}
The coefficient is measurable before $f(p)$ is revealed and the increment has conditional mean zero.  Conditional Doob gives \eqref{eq:lambda-two-doob}.  Lemma \ref{lem:rademacher-hypercontractive} gives
\begin{equation*}
        \E\left[\left|\sum_{\substack{n\le z\\a<P(n)\le b}}f(n)\right|^4\right]
        \le\left(\sum_{\substack{n\le z\\a<P(n)\le b}}
        \mu^2(n)\tau_3(n)\right)^2
        \le\left(\sum_{\substack{n\le z\\a<P(n)\le b}}\tau_3(n)\right)^2.
\end{equation*}
In both cases the sums vanish for $0<z<1$, so the term is zero when $x_i/y_{j-1}\le1$; otherwise the effective lower endpoint is $\max\{x_i/y_j,1\}$, and the weighted Cauchy--Schwarz calculation above applies verbatim. This is the bound used in \eqref{eq:lambda-two-fourth}, so the remainder of the proof of Lemma \ref{lem:lambda-two} is unchanged.

For $\lambda_\ell^{(3)}$, the Cauchy--Schwarz and fourth-moment calculation in the proof of Proposition \ref{prop:caich-smoothing} is replaced by the same inequality.  Lemma \ref{lem:short-reciprocal-primes} is deterministic and unchanged.  Hence
\begin{equation*}
        \E\left[\lambda_\ell^{(3)}(x_i,y_j;f)^2\right]
        \ll \frac{(\log x_i)^4}{X^2(\log y_0)^2},
\end{equation*}
and the same Markov and union calculation gives the $\lambda_\ell^{(3)}$ estimate in \eqref{eq:auxiliary-events}.

The proof of \eqref{eq:W-moment} applies Minkowski's inequality and the hypercontractive inequality, followed by a deterministic count of primes in short multiplicative intervals.  In the Rademacher model Lemma \ref{lem:rademacher-hypercontractive} yields the required moment bound, and each nonnegative divisor sum acquires the restriction $\mu^2(n)=1$.  It is therefore no larger than the corresponding divisor sum in the Steinhaus proof.  We obtain the same estimate
\begin{equation*}
        \E[W_\ell(x_i;f)^{m_0}]
        \ll_{m_0}\left(\frac{x_i}{\log x_i}\right)^{m_0},
\end{equation*}
so the proof of the $W_\ell$ estimate in \eqref{eq:auxiliary-events} is unchanged.

The estimates for $L_\ell^{(12)}$ and $L_\ell^{(2)}$ are first-moment calculations.  Expanding the squares and using \eqref{eq:rademacher-orthogonality} restricts each diagonal sum to squarefree integers.  All terms are nonnegative, so these sums are bounded by the corresponding Steinhaus diagonal sums.  Consequently the estimates
\begin{equation*}
        \E[L_\ell^{(12)}(x_i;f)]
        \ll L\exp(-\ell^{99K}/2)
\end{equation*}
and
\begin{equation*}
        \E[L_\ell^{(2)}(x_i;f)]
        \ll \frac{J}{X\log y_0}
\end{equation*}
hold without change.  The Markov and union arguments following \eqref{eq:L12-expectation} and \eqref{eq:L2-expectation} give their summability estimates.  This proves all the assertions in \eqref{eq:auxiliary-events} for the Rademacher model.
\end{proof}

\subsection{The Rademacher low moment estimate}

For $Y\ge3$, define
\begin{equation}\label{eq:rademacher-euler-product}
        F_Y^{\mathrm R}(s)=\prod_{p\le Y}
        \left(1+\frac{f(p)}{p^s}\right).
\end{equation}
This finite Euler product is the Dirichlet polynomial of $f$ restricted to squarefree $Y$-smooth integers.  We first record the form of Harper's shifted estimate that is needed below.

\begin{lem}[{Harper \cite[Section~4.4]{HarperLow}}]\label{lem:harper-shifted-rademacher}
Let $x$ be sufficiently large, let $2/3\le q\le1$, and let $k,N\in\Z$ satisfy
\begin{equation*}
        0\le k\le\lfloor\log\log\log x\rfloor,
        \qquad |N|\le(\log\log x)^2.
\end{equation*}
Let
\begin{equation*}
        F_{k,x}^{\mathrm R}(s)=
        \prod_{p\le x^{e^{-(k+1)}}}
        \left(1+\frac{f(p)}{p^s}\right).
\end{equation*}
Then
\begin{multline}\label{eq:harper-shifted-full}
\E\left[\left(
        \int_{\substack{|t-N|\le1/2\\
        |t|>(\log\log x)^{-1/2}}}
        \left|F_{k,x}^{\mathrm R}
        \left(\tfrac12-\tfrac{k}{\log x}+it\right)\right|^2\dee t
        \right)^q\right]\\
\ll
        (1+|N|)^{q/4}
        \left(\frac{\log x}{e^k}
        \min\left\{1,
        \frac{1}{(1-q)\sqrt{\log\log x}}\right\}\right)^q.
\end{multline}
When $q=1$, the second entry in the minimum is understood to be $+\infty$.  The implied constant is absolute on this range of parameters.
\end{lem}

\begin{proof}
This is the shifted unit-interval estimate obtained in the Rademacher upper-bound argument of Harper \cite[Section 4.4]{HarperLow}.  It is the estimate given there after the words ``it will suffice to show that''.  Harper first treats $N=0$ through the Rademacher analogues of his Key Propositions 3 and 4.  The final paragraph of that section explains the changes for $1\le |N|\le(\log\log x)^2$ and shows that translating the interval introduces the factor $(1+|N|)^{q/4}$ appearing above.  Harper's product at level $k$ is $F_{k,x}^{\mathrm R}(s)$ as defined above.  The small interval $|t|\le(\log\log x)^{-1/2}$ is excluded in the Rademacher argument because multiplication of the prime signs by $p^{-it}$ does not preserve their distribution.
\end{proof}

\begin{lem}\label{lem:rademacher-shifted}
For $N\in\Z$ let
\begin{equation*}
        J_N(Y)=\frac{1}{\log Y}
        \int_{N-1/2}^{N+1/2}|F_Y^{\mathrm R}(\tfrac12+it)|^2\dee t.
\end{equation*}
Uniformly for $N\in\Z$ and $Y\ge3$,
\begin{equation}\label{eq:rademacher-shifted}
        \E[J_N(Y)^{2/3}]
        \ll(1+|N|)^{1/6}(\log\log Y)^{-1/3}.
\end{equation}
\end{lem}

\begin{proof}
Take $x=Y^e$, $k=0$, and $q=2/3$ in Lemma \ref{lem:harper-shifted-rademacher}.  Then $x^{e^{-1}}=Y$, so $F_{0,x}^{\mathrm R}(s)=F_Y^{\mathrm R}(s)$.  Since $1-q=1/3$ and $\log\log x=\log\log Y+1$,
\begin{equation*}
        \min\left\{1,
        \frac{1}{(1-q)\sqrt{\log\log x}}\right\}
        \ll(\log\log Y)^{-1/2}.
\end{equation*}
After division by $(\log Y)^q$, \eqref{eq:harper-shifted-full} gives
\begin{equation}\label{eq:harper-rademacher-shift}
\E\left[\left(
        \frac{1}{\log Y}
        \int_{\substack{|t-N|\le1/2\\|t|>\delta_Y}}
        |F_Y^{\mathrm R}(\tfrac12+it)|^2\dee t
        \right)^q\right]
\ll(1+|N|)^{q/4}(\log\log Y)^{-q/2},
\end{equation}
where $\delta_Y=(\log\log Y)^{-1/2}$, for $|N|\le(\log\log Y)^2$.  Here $\log\log x=\log\log Y+1$.  The cutoff $(\log\log x)^{-1/2}$ in Lemma \ref{lem:harper-shifted-rademacher} is smaller than $\delta_Y$, so restricting the integral further to $|t|>\delta_Y$ can only decrease it.  Replacing the remaining occurrences of $\log\log x$ by $\log\log Y$ changes only the implied constant.  For all sufficiently large $Y$, one has $\delta_Y<1/2$, so among the integer translates only the interval with $N=0$ meets the omitted neighbourhood $|t|\le\delta_Y$.

When $N=0$, the omitted interval satisfies
\begin{equation*}
\E\left[\left(\frac{1}{\log Y}
        \int_{|t|\le\delta_Y}|F_Y^{\mathrm R}(\tfrac12+it)|^2\dee t
        \right)^q\right]
\le
\left(\frac{1}{\log Y}\int_{|t|\le\delta_Y}
        \E[|F_Y^{\mathrm R}(\tfrac12+it)|^2]\dee t\right)^q.
\end{equation*}
Independence gives, uniformly in $t$,
\begin{equation}\label{eq:rademacher-second-moment-product}
        \E[|F_Y^{\mathrm R}(\tfrac12+it)|^2]
        =\prod_{p\le Y}\left(1+\frac{1}p\right)\ll\log Y.
\end{equation}
The omitted contribution is therefore $O(\delta_Y^q)=O((\log\log Y)^{-q/2})$.  Since $u\mapsto u^q$ is subadditive for $0<q\le1$, this combines with \eqref{eq:harper-rademacher-shift} to give the required estimate for $|N|\le(\log\log Y)^2$.

For $|N|>(\log\log Y)^2$, Jensen's inequality and \eqref{eq:rademacher-second-moment-product} give
\begin{equation*}
        \E[J_N(Y)^q]\le(\E[J_N(Y)])^q\ll1.
\end{equation*}
On this range,
\begin{equation*}
        (1+|N|)^{q/4}(\log\log Y)^{-q/2}\gg1,
\end{equation*}
so the same bound holds.  Taking $q=2/3$ proves \eqref{eq:rademacher-shifted}.  Bounded values of $Y$ are absorbed by enlarging the implied constant.
\end{proof}

\begin{lem}\label{lem:rademacher-low-moment}
Uniformly for $Y\ge3$,
\begin{equation*}
        \E\left[\left(
        \frac{1}{\log Y}\int_{-\infty}^{\infty}
        \left|\frac{F_Y^{\mathrm R}(\tfrac12+it)}{\tfrac12+it}\right|^2\dee t
        \right)^{2/3}\right]
        \ll(\log\log Y)^{-1/3}.
\end{equation*}
\end{lem}

\begin{proof}
Partition the real line into the intervals $[N-1/2,N+1/2]$ with $N\in\Z$.  On each of them, $|\tfrac12+it|^{-2}\ll(1+|N|)^{-2}$.  Since $u\mapsto u^{2/3}$ is subadditive on $[0,\infty)$, Lemma \ref{lem:rademacher-shifted} gives
\begin{align*}
&\E\left[\left(
        \frac{1}{\log Y}\int_{-\infty}^{\infty}
        \left|\frac{F_Y^{\mathrm R}(\tfrac12+it)}{\tfrac12+it}\right|^2\dee t
        \right)^{2/3}\right]\\
&\qquad\ll\sum_{N\in\Z}(1+|N|)^{-4/3}\E[J_N(Y)^{2/3}]\\
&\qquad\ll(\log\log Y)^{-1/3}
        \sum_{N\in\Z}(1+|N|)^{-4/3+1/6}.
\end{align*}
The last exponent is $-7/6<-1$, so the series converges.
\end{proof}

\subsection{The supermartingale and conditional block moment}

For $0\le j\le J$, abbreviate $F_j^{\mathrm R}(s)=F_{y_j}^{\mathrm R}(s)$ and define
\begin{equation}\label{eq:rademacher-Ij}
        I_j^{\mathrm R}=\left(\frac{\log y_j}{\log y_0}\right)^{-1/L}
        \frac{1}{\log y_j}\int_{-\infty}^{\infty}
        \left|\frac{F_j^{\mathrm R}(\tfrac12+it)}{\tfrac12+it}\right|^2\dee t.
\end{equation}

\begin{lem}\label{lem:rademacher-supermartingale}
For all sufficiently large $\ell$, $(I_j^{\mathrm R})_{0\le j\le J}$ is a nonnegative supermartingale with respect to $(\CF_{y_j})_{0\le j\le J}$.  Moreover,
\begin{equation*}
        \E\left[(I_0^{\mathrm R})^{2/3}\right]\ll_KL^{-1/3}.
\end{equation*}
\end{lem}

\begin{proof}
Tonelli's theorem, \eqref{eq:rademacher-second-moment-product}, and Mertens' theorem give
\begin{equation*}
        \E[I_0^{\mathrm R}]
        =\frac{2\pi}{\log y_0}\prod_{p\le y_0}\left(1+\frac{1}p\right)\ll1.
\end{equation*}
Fix $1\le j\le J$ and condition on $\CF_{y_{j-1}}$.  For every $t\in\R$,
\begin{equation*}
        \E\left[\left|1+\frac{f(p)}{p^{1/2+it}}\right|^2\right]
        =1+\frac{1}p.
\end{equation*}
Consequently,
\begin{equation}\label{eq:rademacher-Ij-multiplier}
        \E\left[I_j^{\mathrm R}\mid\CF_{y_{j-1}}\right]
        =b_j^{\mathrm R}I_{j-1}^{\mathrm R},
\end{equation}
where
\begin{equation*}
        b_j^{\mathrm R}=e^{-1/\ell}e^{-1/(\ell L)}
        \prod_{y_{j-1}<p\le y_j}\left(1+\frac{1}p\right).
\end{equation*}
By Lemma \ref{lem:quantitative-mertens}, uniformly in $j$,
\begin{equation*}
\begin{aligned}
        \log b_j^{\mathrm R}
        &=-\frac{1}{\ell}-\frac{1}{\ell L}
        +\sum_{y_{j-1}<p\le y_j}\log\left(1+\frac{1}p\right)\\
        &=-\frac{1}{\ell L}
        +O\left(\sum_{y_{j-1}<p\le y_j}\frac{1}{p^2}\right)
        +O\left(e^{-c\sqrt{\log y_0}}\right)\\
        &=-\frac{1}{\ell L}+o_K\left(\frac{1}{\ell L}\right)<0
\end{aligned}
\end{equation*}
for all sufficiently large $\ell$, uniformly for $1\le j\le J$, including the final block.  Thus $b_j^{\mathrm R}\le1$, and \eqref{eq:rademacher-Ij-multiplier} proves the supermartingale property exactly as in Lemma \ref{lem:supermartingale}.  Finally, Lemma \ref{lem:rademacher-low-moment}, applied with $Y=y_0$, and $\log\log y_0\asymp_KL$ give the stated bound for $I_0^{\mathrm R}$.
\end{proof}

With the sets $\CQ_j$ of smoothness cutoffs from Section \ref{sec:block}, define, for $1\le j\le J$,
\begin{equation*}
        U_j^{\mathrm R}=\frac{1}{\log y_j}\int_0^\infty
        \max_{q\in\CQ_j}|\Psi_f(z,q)|^2\frac{\dee z}{z^2},
\end{equation*}
together with
\begin{equation*}
        U_0^{\mathrm R}
        =\frac{1}{\log y_0}\int_0^\infty
        |\Psi_f(z,y_0)|^2\frac{\dee z}{z^2}
        =\frac{1}{2\pi}I_0^{\mathrm R}.
\end{equation*}

\begin{prop}\label{prop:rademacher-block-moment}
Let $r\ge2$ be fixed.  For all sufficiently large $\ell$, uniformly for $1\le j\le J$,
\begin{equation*}
        \E\left[(U_j^{\mathrm R})^r\mid\CF_{y_{j-1}}\right]
        \ll_r(I_{j-1}^{\mathrm R})^r
        \qquad\text{almost surely}.
\end{equation*}
\end{prop}

\begin{proof}
Fix $j$ and condition on $\CF_{y_{j-1}}$.  Reveal the primes in $(y_{j-1},y_j]$ in increasing order.  If $p$ follows the cutoff $q$, squarefree support gives
\begin{equation}\label{eq:rademacher-cutoff-martingale}
        \Psi_f(z,p)=\Psi_f(z,q)+f(p)\Psi_f(z/p,q).
\end{equation}
The second term has conditional mean zero and its coefficient is measurable before $f(p)$ is revealed.  Thus $q\mapsto\Psi_f(z,q)$ is a martingale along $\CQ_j$, and conditional Doob at exponent $2r$ removes the maximum over $q$, as in \eqref{eq:doob-step}.

For $Z\ge1$, define
$$
        U_{j,Z}^{\mathrm R}
        =
        \frac{1}{\log y_j}\int_1^Z
        \max_{q\in\CQ_j}|\Psi_f(z,q)|^2\frac{\dee z}{z^2}.
$$
Since $\Psi_f(z,q)=0$ for $0<z<1$, we have
$$
        U_{j,Z}^{\mathrm R}\uparrow U_j^{\mathrm R}
        \qquad(Z\to\infty).
$$
Conditional Minkowski, followed by Doob, gives
\begin{equation}\label{eq:rademacher-minkowski}
        \E\left[(U_{j,Z}^{\mathrm R})^r\mid\CF_{y_{j-1}}\right]^{1/r}
        \ll_r\frac{1}{\log y_j}\int_1^Z
        \E\left[|\Psi_f(z,y_j)|^{2r}\mid\CF_{y_{j-1}}\right]^{1/r}
        \frac{\dee z}{z^2}.
\end{equation}
Every squarefree $y_j$-smooth integer factors uniquely as $n=dm$, where $d$ is squarefree and supported on the primes in $(y_{j-1},y_j]$, while $P(m)\le y_{j-1}$.  Hence
\begin{equation*}
        \Psi_f(z,y_j)=
        \sum_{\substack{d\le z\\\mu^2(d)=1\\p\mid d\Rightarrow y_{j-1}<p\le y_j}}
        f(d)\Psi_f\left(\frac{z}{d},y_{j-1}\right).
\end{equation*}
Lemma \ref{lem:rademacher-hypercontractive}, applied conditionally to the new prime signs, yields
\begin{equation}\label{eq:rademacher-hyper-block}
        \E\left[|\Psi_f(z,y_j)|^{2r}\mid\CF_{y_{j-1}}\right]^{1/r}
        \le\sum_{\substack{d\le z\\\mu^2(d)=1\\p\mid d\Rightarrow y_{j-1}<p\le y_j}}
        \tau_{2r-1}(d)\left|\Psi_f\left(\frac{z}{d},y_{j-1}\right)\right|^2.
\end{equation}
Substitute \eqref{eq:rademacher-hyper-block} into \eqref{eq:rademacher-minkowski}, interchange sum and integral by Tonelli, and let $z=du$.  The new-prime divisor factor is
\begin{equation*}
        \sum_{\substack{d\ge1\\\mu^2(d)=1\\p\mid d\Rightarrow y_{j-1}<p\le y_j}}
        \frac{\tau_{2r-1}(d)}d
        =\prod_{y_{j-1}<p\le y_j}
        \left(1+\frac{2r-1}{p}\right) \ll_r1,
\end{equation*}
since $d$ is squarefree and the reciprocal-prime mass of the block is $1/\ell+o(1)$.  What remains is
\begin{equation*}
        \frac{1}{\log y_j}\int_0^\infty
        |\Psi_f(u,y_{j-1})|^2\frac{\dee u}{u^2}.
\end{equation*}
The Dirichlet series of the squarefree $y_{j-1}$-smooth function is $F_{j-1}^{\mathrm R}(s)$.  Lemma \ref{lem:parseval-dirichlet} and \eqref{eq:rademacher-Ij} therefore identify the preceding formula with a multiple of $I_{j-1}^{\mathrm R}$, exactly as in the last part of the proof of Proposition \ref{prop:block-moment}.  The residual factor is exactly $e^{-1/\ell}\exp((j-1)/(\ell L))$, which is at most $2$ uniformly for $1\le j\le J$, including the final block, by \eqref{eq:J-count}.  Thus
\begin{equation*}
        \E\left[(U_{j,Z}^{\mathrm R})^r\mid\CF_{y_{j-1}}\right]^{1/r}
        \ll_r I_{j-1}^{\mathrm R}
\end{equation*}
uniformly in $Z$.  Conditional monotone convergence completes the proof.
\end{proof}

\subsection{Completion of the proof}

\begin{proof}[Proof of the Rademacher case of Proposition \ref{prop:fixed-r}]
We keep the parameters $T(\ell)$ and $T_1(\ell)$ of Section \ref{sec:smoothing} and $H_{\ell,r}$ of Proposition \ref{prop:principal}.  The Rademacher largest-prime decomposition is Lemma \ref{lem:rademacher-decomposition}, the smoothing estimate is Proposition \ref{prop:rademacher-smoothing}, and the stopped Euler-product estimates are given by Lemmas \ref{lem:rademacher-low-moment} and \ref{lem:rademacher-supermartingale}.  Proposition \ref{prop:rademacher-block-moment} gives the conditional block moment.  With these replacements, the stopping argument in the proof of Proposition \ref{prop:principal} gives
\begin{equation}\label{eq:rademacher-principal}
        \sum_{\ell\ge2}\P\left(
        \sup_{X_{\ell-1}<x_i\le X_\ell}\sup_{0\le j\le J}
        M_\ell(x_i,y_j;f)>C_rT_1(\ell)L^{-1/2+1/r}\right)<\infty.
\end{equation}
The stopping event is
\begin{equation*}
        \CS_\ell^{\mathrm R}=
        \left\{\max_{0\le j\le J}I_j^{\mathrm R}
        \le T_1(\ell)^{1/2}L^{-1/2}\right\}.
\end{equation*}
Doob's inequality gives $\sum_\ell\P((\CS_\ell^{\mathrm R})^c)<\infty$.  Let
\begin{equation*}
        \CS_{\ell,j-1}^{\mathrm R}
        =
        \left\{I_{j-1}^{\mathrm R}
        \le T_1(\ell)^{1/2}L^{-1/2}\right\}.
\end{equation*}
Proposition \ref{prop:rademacher-block-moment} and conditional Markov give
\begin{equation*}
\ind_{\CS_{\ell,j-1}^{\mathrm R}}
\P\left(U_j^{\mathrm R}>C_rH_{\ell,r}
\mid\CF_{y_{j-1}}\right)
\ll_rT_1(\ell)^{-r/2}L^{-1}.
\end{equation*}
Summing over $J\ll_KL$ blocks gives a summable bound.  On $\CS_\ell^{\mathrm R}$, the identity $U_0^{\mathrm R}=I_0^{\mathrm R}/(2\pi)$ controls the endpoint $j=0$, proving \eqref{eq:rademacher-principal}.

Combining \eqref{eq:rademacher-principal} with Proposition \ref{prop:rademacher-smoothing} gives, for some $C_*=C_*(r,K,c_0,m_0)$,
\begin{equation}\label{eq:rademacher-variance}
        \sum_{\ell\ge1}\P\left(
        \sup_{X_{\ell-1}<x_i\le X_\ell}
        \frac{V_\ell(x_i;f)}{x_i}>C_*T(\ell)L^{-1/2+1/r}\right)<\infty.
\end{equation}
Here the principal term is bounded by \eqref{eq:rademacher-principal} and the auxiliary terms by \eqref{eq:auxiliary-events}.

It remains to repeat the final martingale argument.  Let $R(x)=(\log\log x)^{1/4+1/(2r)+\eta}$ and, for each test point $x_i$ in the $\ell$th scale, let
\begin{equation*}
        \Sigma_{\ell,i}^{\mathrm R}
        =\left\{V_\ell(x_i;f)
        \le C_*x_iT(\ell)L^{-1/2+1/r}\right\}.
\end{equation*}
Equation \eqref{eq:rademacher-variance} gives
\begin{equation}\label{eq:rademacher-variance-union}
        \sum_{\ell\ge1}\P\left(
        \bigcup_{X_{\ell-1}<x_i\le X_\ell}
        (\Sigma_{\ell,i}^{\mathrm R})^c\right)<\infty.
\end{equation}

Enumerate the primes in $(y_0,y_J]$ as $p_1<\cdots<p_N$, and define
\begin{equation*}
        \CH_0=\CF_{y_0},
        \qquad
        \CH_k=\CF_{y_0}\vee
        \sigma(f(p_h):1\le h\le k),
\end{equation*}
and
\begin{equation*}
        Z_k=f(p_k)\Psi_f'\left(\frac{x_i}{p_k},p_k\right),
        \qquad
        S_k=\left|\Psi_f'\left(\frac{x_i}{p_k},p_k\right)\right|.
\end{equation*}
By the strict smoothness condition, $S_k$ is $\CH_{k-1}$-measurable and $\E[Z_k\mid\CH_{k-1}]=0$.  Since $S_k\le x_i/p_k$, every $Z_k$ is bounded.  Thus $(Z_k)$ is a real martingale difference sequence and
\begin{equation*}
        \sum_{k=1}^NS_k^2=V_\ell(x_i;f),\qquad
        \sum_{k=1}^NZ_k=M_f^{(1)}(x_i).
\end{equation*}
Lemma \ref{lem:stopped-hoeffding} therefore gives
\begin{equation*}
\begin{aligned}
&\P\left(
|M_f^{(1)}(x_i)|\ge\sqrt{x_i}R(x_i),\
\Sigma_{\ell,i}^{\mathrm R}\right)\\
&\qquad\le
2\exp\left(-\frac{R(x_i)^2}
{10C_*T(\ell)L^{-1/2+1/r}}\right).
\end{aligned}
\end{equation*}
By \eqref{eq:L-comparability},
\begin{equation*}
        \frac{R(x_i)^2}{T(\ell)L^{-1/2+1/r}}
        \asymp\frac{L^{1+2\eta}}{T(\ell)}
        =\ell^{K(1+2\eta)-10}.
\end{equation*}
There are at most $\exp(C_{c_0}\ell^K)$ test points in the $\ell$th scale.  Hence a union bound, followed by \eqref{eq:rademacher-variance-union}, gives the bound
\begin{equation*}
\begin{aligned}
&\P\left(
\sup_{X_{\ell-1}<x_i\le X_\ell}
\frac{|M_f^{(1)}(x_i)|}{\sqrt{x_i}R(x_i)}>1\right)\\
&\qquad\le\P\left(\bigcup_{X_{\ell-1}<x_i\le X_\ell}
(\Sigma_{\ell,i}^{\mathrm R})^c\right)
+2\exp\left(C_{c_0}\ell^K-c\ell^{K(1+2\eta)-10}\right).
\end{aligned}
\end{equation*}
The second term is summable because $2K\eta>10$, and the first is summable by \eqref{eq:rademacher-variance-union}.  The smooth contribution has summable exceptional probabilities by the argument following Lemma \ref{lem:rademacher-decomposition}.  Thus, almost surely,
\begin{equation*}
        |M_f(x_i)|\le2\sqrt{x_i}R(x_i)
\end{equation*}
for every sufficiently large test point $x_i$.  Proposition \ref{prop:test-point-reduction} extends this estimate to all sufficiently large $x$.  This proves the Rademacher case of Proposition \ref{prop:fixed-r}, and hence completes the proof of Theorem \ref{thm:main}.
\end{proof}

\section*{Declaration on the use of AI}
All maths contained in this paper was calculated and written by humans. The final draft of the paper was passed through ChatGPT 5.6 Sol for helping to polish the exposition.

\end{document}